\documentclass[11pt]{article}
\usepackage[final]{epsfig}
\usepackage[left=3cm,right=3cm, top=2.5cm,bottom=2.5cm,bindingoffset=0cm]{geometry}
\usepackage{graphics}
\usepackage{amsmath}
\usepackage{amsfonts}
\usepackage{latexsym}
\usepackage{amssymb, mathabx}
\usepackage{amsthm}
\usepackage{graphicx}
\usepackage{epstopdf}
\usepackage{multicol,multirow}
\usepackage{hyperref, enumitem}
\usepackage{mathtools}

\usepackage{cases}

\usepackage[numbers, sort&compress]{natbib}

 \usepackage{relsize}

\usepackage[nottoc]{tocbibind}

\usepackage{amsthm}

\usepackage{tikz}
\usetikzlibrary{positioning}
\usetikzlibrary{decorations.text}
\usetikzlibrary{decorations.pathmorphing}
\usetikzlibrary{decorations.markings}

\usetikzlibrary{arrows} 
\tikzset{
    >=stealth',
    pil/.style={
           ->,
           thick,
           shorten <=2pt,
           shorten >=2pt,}
}
\tikzset{->-/.style={decoration={
  markings,
  mark=at position .7 with {\arrow{>}}},postaction={decorate}}}
  \tikzset{a/.style={decoration={
  markings,
  mark=at position .52 with {\arrow{angle 90}}},postaction={decorate}}}
\tikzset{-<-/.style={decoration={
  markings,
  mark=at position .4 with {\arrow{<}}},postaction={decorate}}}

\makeatother

\newcounter{AppCounter}

\def\restrict#1{\raise-.5ex\hbox{\ensuremath|}_{#1}}

\newtheorem{lemma}{Lemma}[section]
\newtheorem{proposition}[lemma]{Proposition}

\newtheorem{remark-definition}[lemma]{Remark-Definition}
\newtheorem{theorem}[lemma]{Theorem}
\newtheorem{corollary}[lemma]{Corollary}

\newtheorem{proposition-conjecture}[lemma]{Proposition-conjecture}

\theoremstyle{definition}
\newtheorem{example}[lemma]{Example}

\newtheorem{definition}[lemma]{Definition}
\newtheorem{remark}[lemma]{Remark}

\newcommand{\proofend}{\hfill$\Box$\bigskip}

\newcommand{\R}{{\mathbb R}}

\newcommand{\G}{\mathcal G}
\newcommand{\B}{  B}
\newcommand{\metric}{\langle \,, \rangle}       
\renewcommand{\div}{\mathrm{div} \,}

\newcommand{\SDiff}{\mathrm{SDiff}}
\newcommand{\DSDiff}{\mathrm{DSDiff}}

\newcommand{\vect}{\Vect}

\newcommand{\dsvect}{\mathrm{DSVect}}

\newcommand{\mdiff}{\mathrm{MDiff}}
\newcommand{\mvect}{\mathrm{MVect}}
\newcommand{\mdens}{\mathrm{MDens}}
\newcommand{\gdiff}{\mathrm{GDiff}}
\newcommand{\gvect}{\mathrm{GVect}}
\newcommand{\gdens}{\mathrm{GDens}}
\newcommand{\Dens}{\mathrm{Dens}}
\usepackage{bm}
\newcommand{\bphi}{{\pmb{\phi}}}
\newcommand{\bu}{{\pmb{u}}}
\newcommand{\bv}{{\pmb{v}}}
\newcommand{\bF}{{\pmb{f}}}
\newcommand{\bg}{{\pmb{g}}}
\newcommand{\bmu}{{\pmb{\mu}}}
\newcommand{\bxi}{{\pmb{\xi}}}
\newcommand{\balpha}{{\pmb{\alpha}}}

\newcommand{\vol}{\mathrm{vol}_M}
\newcommand{\volA}{\mathrm{vol}_A}

\newcommand{\bpsi}{{\pmb{\psi}}}

\newcommand{\bff}{{\pmb{f}}}

\newcommand{\Ker}{\mathrm{Ker}}

\newcommand{\grad}[1]{\nabla #1}

\newcommand{\diff}[1]{{d}  #1}

\newcommand{\T}{{T}}
\newcommand{\Cont}{{C}}
\newcommand{\LieBracket}{ [\, , ] }

\newcommand{\g}{\mathfrak{g}}

\newcommand{\id}{\mathrm{id}}

\newcommand{\VS}{{\mathrm{VS}}}

\newcommand{\W}{\textnormal{Dens}} 
\newcommand{\Vect}{\mathrm{Vect}}
\newcommand{\Diff}{\textnormal{Diff}} 
\newcommand{\Sheet}{\Gamma}
\newcommand{\Src}{\mathrm{src}}
\newcommand{\Trg}{\mathrm{trg}}
\newcommand{\units}{{\id}}
\newcommand{\Dom}{D}
\newcommand{\chiplus}{\chi^+_\Sheet}

\newcommand{\chimin}{\chi_\Sheet^-}

\newcommand{\Dompm}{\Dom_\Sheet^\pm}
\renewcommand{\L}{\mathcal L}

\newcommand{\A}{\mathcal{A}}
\newcommand{\I}{\mathcal{I}}
\renewcommand{\H}{\mathcal{H}}
\newcommand{\F}{\mathcal{F}}
\renewcommand{\P}{\mathcal{P}}

\newcommand{\low}[1]{\raise-.0ex\hbox{$\scriptstyle #1$}}
\newcommand{\high}[1]{\raise.5ex\hbox{$\scriptstyle #1$}}

\renewcommand{\tfrac}{\frac}

\usepackage{autonum}

\newcommand{\marginnote}[1]
{%
}

\newcounter{ai}

\newcounter{bk}

\title {Geometry of generalized fluid flows}

\author{Anton Izosimov\thanks{
Department of Mathematics,
University of Arizona;
e-mail: {\tt izosimov@math.arizona.edu}
} \,
and Boris Khesin\thanks{
Department of Mathematics,
University of Toronto;
e-mail: \tt{khesin@math.toronto.edu}
} }

\date{}

\begin{document}

\maketitle

\begin{abstract}
The Euler equation of an ideal (i.e. inviscid incompressible) fluid can be regarded, following V.Arnold, as  the geodesic flow of the right-invariant 
$L^2$-metric on the group of volume-preserving diffeomorphisms of the flow domain. 
In this paper we describe the common origin and symmetry of  generalized flows, multiphase fluids (homogenized vortex sheets), and conventional vortex sheets: 
they all correspond to geodesics on certain groupoids of multiphase diffeomorphisms. Furthermore, we prove that all these problems are Hamiltonian with respect to a Poisson structure on a dual Lie algebroid, generalizing the Hamiltonian property of the Euler equation on a Lie algebra dual. 
\end{abstract}

\tableofcontents

\section{Introduction} \label{intro}

Classical hydrodynamics deals with an ideal (i.e. inviscid incompressible) fluid, whose motion is described by the Euler equation. 
In this paper we consider a  broader setting of multiphase fluids and generalized flows. A {\it multiphase fluid} consists of several fractions that can freely penetrate through each other without resistance and are constrained only by the conservation of total density. Such flows arise, in particular,  in connection with vortex sheets in an ideal fluid, i.e. hypersurfaces of discontinuity in fluid velocity with different speed of fluid layers 
on different sides of the hypersurface. By relaxing the condition of a sharp border between the layers 
one obtains \textit{homogenized  vortex sheets} \cite{brenier}, which allow mixing 
of the two parts of the fluid, rather than separating them by a hypersurface. Such homogenized  vortex sheets can be thought of as examples of multiphase flows. Beyond the vortex sheet setting, multiphase fluids arise e.g. in plasma physics and chemistry. 

\begin{figure}[b]
\centering
\includegraphics[width=4.5in]{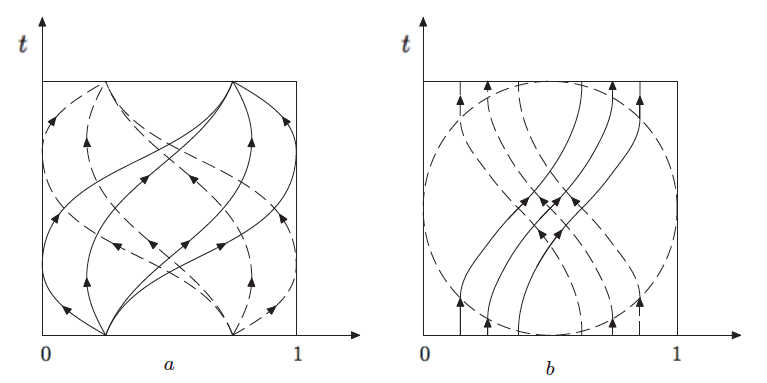}
\vspace{-7pt}
\caption{\small  Trajectories of particles in one-dimensional analogues of generalized flows corresponding to (a) continuum of phases for the flip of the
interval [0, 1] and (b) a multiphase flow with two phases for the interval-exchange map $[0,1/2]\leftrightarrow [1/2,1]$; see \cite{AK}.}
\vspace{-10pt}
\label{Fig.0}
\end{figure}

Of particular interest are multiphase fluids with continuum of phases (or \textit{generalized flows}), introduced by Y.\,Brenier \cite{brenier2}. One can think of them as flows in which every fluid particle 
spreads into a cloud thus moving to any other point of the manifold with certain probability (we define this precisely below), see Figure \ref{Fig.0}. While, according to A.\,Shnirelman \cite{Sh}, a shortest curve on the group of volume-preserving diffeomorphisms does not exist between some pairs of maps, generalized flows of Brenier do allow such a shortest solution for a large class of diffeomorphisms.

In this paper we describe the common origin and symmetry of both multiphase fluids (equivalently, homogenized vortex sheets) and generalized flows (fluids with continuum of phases): they both correspond to geodesics on certain groupoids of multiphase diffeomorphisms. Groupoids can be thought of as groups with partially defined multiplication. We also present the Hamiltonian framework for them by describing the 
corresponding dynamics as Euler-Arnold flows  for right-invariant energy metrics on the groupoid. In other words, we prove that generalized flows are Hamiltonian for the corresponding  Poisson structure on the dual Lie algebroid, generalizing Lie-Poisson structures.

Recall that in  1966 Arnold  proved that the Euler equation for an ideal fluid describes the geodesic flow of a right-invariant metric on the group of volume-preserving 
diffeomorphisms of  the flow domain \cite{Arn66}.
This insight turned out to be indispensable for the study of geometry and topology of fluid flows, Hamiltonian properties
and conservation laws in hydrodynamics, 
as well as a powerful tool for obtaining sharper existence and uniqueness results for Euler-type equations~\cite{AK}. 
However, such objects as the above-mentioned multiphase fluids or generalized flows do not fit  into Arnold's approach.
In the  paper \cite{IK} on classical vortex sheets in incompressible flows we introduced the language of Lie groupoids in hydrodynamics. 
In the present paper we demonstrate its universality by extending  Arnold's framework to other Lie groupoids with one-sided invariant metrics, thus treating generalized flows
 (which did not allow any group interpretation before) and vortex sheets on the same footing, as well as developing
a  groupoid-theoretic description for many fluid dynamical settings.

\subsection{Groupoid framework for generalized flows}

Recall that  the hydrodynamical Euler equation for an   incompressible fluid filling a closed compact Riemannian manifold $M$
 is the following evolution law on the velocity field $u$:
\begin{equation}
\partial_t u+\nabla_u u=-\nabla p\,,
\end{equation}
supplemented by the divergence-free condition ${\rm div}\, u=0$ on $M$. The pressure function $p$ is defined uniquely modulo an additive constant by those conditions. This setting also extends to manifolds with boundary, as well as non-compact manifolds (such as $\R^n$), by imposing appropriate boundary or decay conditions.
Arnold's theorem sheds light on the origin of this equation:

\begin{theorem}[Arnold \cite{Arn66}]
The Euler equation can be regarded as an equation of the geodesic flow on the group $\SDiff(M)$ 
of volume-preserving diffeomorphisms of $M$ with respect to the right-invariant 
 metric given at the identity of the group by the squared $L^2$-norm of the fluid's velocity field (i.e., the fluid kinetic energy\footnote{The 
$L^2$-metric is twice the kinetic energy of the fluid, which leads to a simple time rescaling, and we will not be mentioning this throughout 
the paper.}).
\end{theorem}

This theorem provides an attractive way to construct Euler solutions as shortest curves, i.e. geodesics, joining two volume-preserving diffeomorphisms of $M$. However, in \cite{Sh} Shnirelman proved that not all pairs of such diffeomorphisms admit a
shortest curve connecting them. This variational problem was ``cured'' by Brenier \cite{brenier2}, who introduced the space of generalized fluid flows and proved the existence in that space of a shortest curve joining any two volume-preserving diffeomorphisms from a large class.

Generalized flows satisfy the following equations:
{
\begin{subnumcases}{\label{eq:gf}}
\partial_t (\rho_a u_a) + \div(\rho_a u_a \otimes u_a) + \rho_a \grad p = 0\,,\label{eq:gf1} \\
\partial_t \rho_a + \div(\rho_a u_a ) = 0\,, \label{eq:gf2} 
\end{subnumcases}
along with the constraint $\int_A \rho_a \,da = 1$.}
Here $u_a \in \Vect(M)$ is the fluid velocity field, depending on an additional parameter $a$ belonging to a certain measure space $A$. One can think of $A$ as enumerating fractions of the fluid, with $u_a$ being the velocity of a particular fraction. Likewise $\rho_a \in C^\infty(M)$ is the mass density of the fraction with label $a \in A$. The pressure function $p \in \Cont^\infty(M)$ is common for all fractions.

{

\begin{remark}
Using \eqref{eq:gf2} one can rewrite \eqref{eq:gf1}  in the form similar to the classical Euler equation:
 \begin{equation}\label{eq:gfeuler}
 \partial_t u_a + \nabla_{u_a}u_a = -\grad p.
 \end{equation}
 The above form is given for consistency with \cite{brenier2}, and it also simplifies the derivation of equation for the pressure function.

{
 Namely, the pressure can be obtained} 
from the velocities $u_a$ and densities $\rho_a$ as follows. Integrating \eqref{eq:gf2} over the space $A$ we get the condition \begin{equation}\label{eq:nde}\div(\int_A \rho_a u_a da) = 0,\end{equation} which can be thought of as an analog of the condition $\div u = 0$ for the classical Euler equation.  Further, taking the divergence of\eqref{eq:gf1}, integrating over $A$, and using \eqref{eq:nde} we get
\begin{equation}\label{eq:pe}
\Delta p = -\div \int_A  \div(\rho_a u_a \otimes u_a) \,da,
\end{equation}
which is a Poisson equation and hence has a unique solution for the pressure function, up to an additive constant.
\end{remark}
}

\begin{theorem}[= Theorem \ref{thmMain4}]\label{thmIntro1}
The Euler equations \eqref{eq:gf} for a generalized flow are geodesic
equations for the right-invariant $L^2$-metric on (source fibers of) the Lie groupoid $\gdiff(M)$ 
of generalized diffeomorphisms. Equivalently, the Euler equations \eqref{eq:gf} are  the groupoid  Euler-Arnold equations  corresponding to the 
$L^2$-metric on the algebroid $\gvect(M)$. \end{theorem}

The Lie groupoid $\gdiff(M)$ of generalized diffeomorphisms is {a natural generalization of the group $\SDiff(M)$ of volume-preserving diffeomorphisms. (Just like the latter arises from ``integrating'' the condition $\div u =0$ on fluid velocities, the groupoid $\gdiff(M)$ ``integrates'' equation \eqref{eq:nde}.) The definition of that groupoid is as follows}. Its base is the space $\gdens(M)$ of  \textit{generalized densities}, i.e. sets of densities $ \bmu:=\{\mu_a \in \Dens(M) \mid a\in A\}$ such that all $\mu_a$ are positive, have prescribed masses 
 $c_a$, i.e. \begin{equation}\label{eq:ca} \int_M  \mu_a = c_a,\end{equation} and together constitute the fixed volume form $\vol$ on $M$, i.e. $\int_A \mu_a \,da = \vol$ at each point of $M$ {(in particular, $\int_A c_a da = \int_M \vol$)}. 
One can think of those densities as a set $A$ of different  
fractions of an incompressible fluid, penetrating through each other without resistance. Such a generalized density $\bmu$ can also be interpreted as a doubly stochastic measure $\mu_a \wedge da$ on the direct product $M \times A$. { The relation between densities $\mu_a$ and functions $\rho_a$ introduced above is $\mu_a = \rho_a \vol$. In particular, the condition $\int_A \mu_a \,da = \vol$ is equivalent to {the constraint $\int_A \rho_a \,da = 1$}.}

The elements of  $\gdiff(M)$ are triples $(\bphi\,;\bmu, \bmu')$ where $\bphi:=\{\phi_a \in \Diff(M) \mid a\in A\}$
is a \textit{generalized diffeomorphism}, and  $\bmu, \bmu'\in \gdens(M)$ are generalized densities such that  $\bphi_*\bmu=\bmu'$ component-wisely, i.e. $\phi_{a*}\mu_a=\mu'_a$ for each $a\in A$. 
The multiplication of such triples is defined by the natural composition, $(\bpsi\,;\bmu', \bmu'')(\bphi\,;\bmu, \bmu') := (\bpsi\bphi\,;\bmu, \bmu'')\,.$

The  infinitesimal object corresponding to this Lie groupoid is the Lie algebroid $\gvect(M)$
 describing the space of velocities for a generalized fluid. It is a vector bundle over $ \gdens(M)$ with the following 
structure. Its fiber over $\bmu  \in \gdens(M)$ is the space $\gvect(M, \bmu)$ that  consists
 of \textit{generalized vector fields} on $M$ of the form 
$\bu := \{ u_a \mid a\in A\} $ with  $u_a\in \vect(M)$ that are ``divergence-free" with respect to the generalized volume form: $\int_A \L_{u_a}\mu_a \,da=0$ { (the latter equation is equivalent to \eqref{eq:nde}).} The vector bundle  $\gvect(M)$ carries additional structures, namely a bracket on sections and a so-called anchor map, see Section \ref{sec:GF}. These structures endow the dual bundle $\gvect(M)^*$ with a Poisson structure. Equations \eqref{eq:gf} are Hamiltonian with respect to that structure:

\begin{theorem}[=Theorem \ref{thmMain3}]\label{thmIntro2}
The Euler equations \eqref{eq:gf} for a generalized flow 
written on the dual  $\gvect(M)^*$ of the algebroid are Hamiltonian with respect to 
the natural Poisson structure on  the dual algebroid and the Hamiltonian function given by the $L^2$ kinetic energy.
\end{theorem}

The above two theorems provide the group-theoretic and Hamiltonian frameworks for generalized flows.

\begin{remark}The 
smoothness of the groupoid and algebroid is understood below 
in the Fr\'{e}chet $C^\infty$ setting. Similarly, one can consider the setting of Hilbert manifolds modeled on Sobolev $H^s$ spaces for sufficiently large $s$, 
$s>\dim M/2 +1$, 
cf. \cite{EbMars70}.\end{remark}
{
\begin{remark}
Theorem \ref{thmIntro2} remains valid if we exclude the condition \eqref{eq:ca} from the definition of the groupoid. That condition is added for technical reasons (specifically, to make the groupoid \emph{transitive}, see Definition \ref{def:trans} below) and does not affect equations \eqref{eq:gf}. Indeed, preservation of masses $ \int_M  \mu_a$ is just a consequence of those equations.
\end{remark}
}

\subsection{Groupoid setting for multiphase fluids}

In this section we discuss the ``discrete version" of generalized flows, namely, multiphase flows on a Riemannian  manifold $M$. Such flows appear in \cite{brenier} in the context of homogenized vortex sheets and are governed by the following equations:
\begin{align}\label{eq:hvse}
\begin{cases}
\partial_t u_j + \nabla_{u_j}u_j = -\grad p\,, \\ 
  \partial_t \rho_j + \div(\rho_j u_j) = 0\,.
\end{cases}
\end{align}
Here $\rho_1, \dots, \rho_n\in C^\infty(M)$ are mass densities of $n$ phases of the fluid subject to { the constraint} $\sum_{j=1}^n \rho_j = 1$, the vector fields $u_1, \dots, u_n\in \Vect(M)$ are the corresponding fluid velocities, and the pressure $p \in \Cont^\infty(M)$ is common for all phases. { These equations can be thought of as a discrete analogue of \eqref{eq:gf}, which becomes particularly transparent upon rewriting equation  \eqref{eq:gf1} in the form \eqref{eq:gfeuler}. Conversely, we can rewrite the first equation in \eqref{eq:hvse} in the form 
$$
\partial_t (\rho_j u_j) + \div(\rho_j u_j \otimes u_j) + \rho_j \grad p = 0.
$$
Furthermore, the second equation implies 
\begin{equation}\label{eq:sde}
\div \sum_{j=1}^n \rho_j u_j = 0,
\end{equation}
which results in the following equation for the pressure, cf. \eqref{eq:pe}:
$$
\Delta p = -\div \sum_{j=1}^n  \div(\rho_i u_i \otimes u_i).
$$}

 The Lie groupoid $\mdiff(M)$ underlying equations \eqref{eq:hvse} is a discrete version of the groupoid $\gdiff(M)$. Its base is the space $\mdens(M)$ of  \textit{multiphase densities}, i.e. $n$-tuples of densities 
$ \bmu:=(\mu_1,...,\mu_n)$ satisfying the conditions that all densities $\mu_i$ are positive and sum to a fixed density $\vol$ everywhere on $M$, 
while their total masses are given by a fixed $n$-tuple of constants $c_1, \dots, c_n \in \R$. 
These densities can be thought of as densities of different mutually penetrating fractions of the fluid,  
subject only to the total incompressibility condition. %

Now the elements of our Lie groupoid $\mdiff(M)$ are $n$-tuples of diffeomorphisms of $M$ preserving the property of incompressibility of multiphase densities,
i.e. the set of tuples $(\bphi\,;\bmu, \bmu'):=(\phi_1, ..., \phi_n;\mu_1, ..., \mu_n, \mu'_1, ..., \mu'_n)$ with multiphase forms $\bmu, \bmu'\in \mdens(M)$ such that the 
multiphase diffeomorphism $\bphi$ push-forwards one of them to the other, $\bphi_*\bmu=\bmu'$ component-wisely. 
The multiplication in $\mdiff(M)$ is defined in the same way as for $\gdiff(M)$.

The corresponding Lie algebroid $\mvect(M)$  
is the space of possible velocities of the multiphase fluid. It is a vector bundle over $ \mdens(M)$ where
the fiber of $\mvect(M)$ over $\bmu \in \mdens(M)$ is the space $\mvect(M, \bmu)$ which consists
 of \textit{multiphase vector fields} on $M$ ``divergence-free" with respect to the multiphase volume form, i.e. vector fields of the form 
$\bu := (u_1,...,u_n)$, where 
 $u_{i}\in \vect(M)$ are such that  $\sum\nolimits_{j=1}^n \L_{u_j}\mu_j=0 $.

\begin{theorem}[=Theorem \ref{thmMain2}]\label{thmIntro3} 
The Euler equations \eqref{eq:hvse} for a multiphase fluid flow are geodesic
equations for the right-invariant $L^2$-metric on (source fibers of) the Lie groupoid $\mdiff(M)$ 
of multiphase volume-preserving diffeomorphisms. Equivalently, they are  groupoid  Euler-Arnold equations  corresponding to the 
$L^2$-metric on the algebroid $\mvect(M)$.  \end{theorem}

For the case of a flat space $M$ the geodesic (although not the group) nature of homogenized vortex sheets (i.e. multiphase flows) was established in  \cite[Proposition 6]{Loesch}.
One can  see that the standard hydrodynamical Euler equation is a particular case of the above multiphase equations 
with only one phase, $n=1$. Furthermore,  equations~\eqref{eq:hvse}   can be described within the Hamiltonian framework:

\begin{theorem}[=Theorem \ref{thmMain}]\label{thmIntro4}
The Euler equations \eqref{eq:hvse} for a multiphase flow 
written on the dual  $\mvect(M)^*$ of the algebroid are Hamiltonian with respect to 
the natural Poisson structure on  the dual algebroid and the Hamiltonian function given by the $L^2$ kinetic energy.
\end{theorem}

This theorem is an analogue of the Hamiltonian property of the 
Euler-Arnold equation  on the dual to a Lie algebra with respect to the Lie-Poisson structure.

Return to the metric properties of the groupoid Euler-Arnold equation.  
Given any initial density  $\bmu \in \mdens(M)$, consider the subset $\mdiff(M)_\bmu\subset \mdiff(M)$ of multiphase diffeomorphisms which push $\bmu$ forward to another multiphase density (a so-called \textit{source fiber} of the groupoid $\mdiff(M)$). 
That set is equipped with an $L^2$-metric. At the same time, there is a natural metric $\metric_{\mdens}$ on  the space $\mdens(M)$ of multiphase densities induced by the well known Wasserstein metric. The connection between those two metrics is described by the following result.

\begin{theorem}[=Theorem \ref{geodDescription}]\label{thmIntro5}
For any multiphase density  $\bmu \in \mdens(M)$  the groupoid target mapping 
$\Trg \colon (\mdiff(M)_\bmu, \metric_{L^2}) \to (\mdens(M), \metric_\mdens)$ is a Riemannian submersion.
In particular, horizontal geodesics on  $\mdiff(M)_\bmu$ project to geodesics on  $\mdens(M)$. Those geodesics correspond to potential solutions of the system  \eqref{eq:hvse}.
 \end{theorem}
\begin{figure}[t]
\centerline{
\begin{tikzpicture}[thick, scale = 1.2]
\draw [thin] (0,-1.) -- (5,-1.);
  \fill [opacity = 0.05] (-0.2,3) -- (5.2,3) -- (5.2,0.5) -- (-0.2,0.5) -- cycle;
\draw [very thin] (0,0.7) -- (0,2.8);
\draw [very thin] (1,0.7) -- (1,2.8);
\draw [] (2,0.7) -- (2,2.8);
\draw [very thin] (3,0.7) -- (3,1.5);
\draw [very thin] (4,0.7) -- (4,1.5);
\draw [very thin] (5,0.7) -- (5,1.5);
\draw [very thin] (4,2.8) -- (4,1.5);
\draw [very thin] (5,2.8) -- (5,1.5);
\draw [very thin] (3,2) -- (3,2.3);
\draw [very thin] (3,2.7) -- (3,2.8);
\draw  [very thick, ->] (2,1.5) -- (2.5,1.5);
   \fill (2,1.5) circle [radius=1.5pt];
      \fill (2,-1.) circle [radius=1.5pt];
      \node  at (1.7,1.5) () {$\id_\bmu$};
      \draw [thin] (2,1.35) -- (2.15,1.35) -- (2.15,1.5);
           \node []  at (2.7,1.75) () {$\bu =  \grad \bF$};
          \node  at (2,-1.2) () {$\bmu$};
                 \node  at (2.8,-0.75) () {$\bxi = -\L_\bu \bmu$};
          \draw  [very thick, ->] (2,-1) -- (2.5,-1);
                   \draw  [->] (2,0.2) -- (2,-0.5);   
                             \node  at (2.3,-0.1) () {$\Trg$};    
                                       \node  at (0.2,-1.25) () {$\mdens(M)$};  
                                          \node  at (0.2,0.25) () {$\mdiff(M)_\bmu$};      
                                                                                    \node  at (2.9,2.5) () {$\SDiff_{\bmu}(M)$};      
\end{tikzpicture}
}
\caption{Riemannian submersion for the groupoid. Here $\SDiff_\bmu(M) := \{ \bphi \mid \phi_i^* \mu_i = \mu_i\}$ is the group of volume-preserving multiphase diffeomorphisms, and $\bu$ is a horizontal vector field projecting to $\bxi = -\L_\bu \bmu$. The latter can be regarded as the velocity of the multiphase density $\bmu$, and $\langle \bxi, \bxi \rangle_\mdens = \langle \bu, \bu \rangle_{L^2}$. }\label{fig:submersion}
\end{figure}
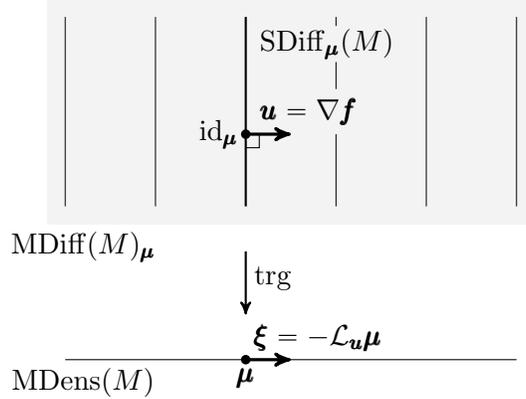

In particular, this result implies a geodesic description of potential solutions to \eqref{eq:hvse}, cf. \cite[Proposition 7]{Loesch}. 
These potential solutions have the form $\bu =  ( \grad f_1,...,   \grad f_n)$ in $M$, see Figure~\ref{fig:submersion}. 
The Wasserstein-type  metric $\metric_\mdens$ is apparently related to
the metric between vector densities described recently in \cite{BV, ciosmak2021optimal}.

{

One of the byproducts of the groupoid approach is the following generalized Kelvin's theorem.
Namely, define the multiphase vorticity $\pmb{\omega}:=\diff \bu^\flat$ for an $n$-tuple of vector fields $\bu \in \mvect(M)$ as
the component-wise vorticity $n$-tuple, i.e. $\omega_j := d u_j^\flat$ with  $u_j^\flat$ standing for the $1$-forms metric-dual to the vector fields $u_j$.

\begin{corollary}[=Corollary \ref{cor:singKelvin}]
For a multiphase fluid the vorticity is ``frozen into the flow" in the generalized sense:
$ \partial_t \pmb{\omega} + \L_{\bu} \pmb{\omega} = 0$, that is the vorticity of each phase is transported by the corresponding velocity field: $\partial_t \omega_j + \L_{u_j} \omega_j = 0\,.$
\end{corollary}

We would like to emphasize that in the classical Euler equation, the vorticity (along with circulations in a non-simply-connected $M$) fully determine the velocity field. In the multiphase setting the situation is different: in particular, there are nontrivial solutions with zero vorticity and zero circulation. The reason is in the different geometry of  symplectic leaves 
for the corresponding Poisson bracket. Indeed, in the group setting these leaves are coadjoint orbits in $\g^*$ of the corresponding group, while in the case of an algebroid one has the group action only in the kernel bundle of the corresponding  anchor map (and in its dual bundle). The corresponding symplectic leaves are obtained by taking the inverse images of the orbits for the action under the projection of $A^*$ into that dual bundle, cf.  \cite{IK}. We hope to return to this description
in a future publication.
}

\subsection{Structure of the paper}

The rest of the paper is a detailed discussion of objects involved in the above theorems, along with proofs of those theorems. We start with the discrete case (Theorems \ref{thmIntro3} - \ref{thmIntro5}). It is discussed in Sections \ref{sec:kinematics} and \ref{sect:dyn_vs} (while in Section \ref{sec:gen} we recall basics of the groupoid and algebroid theory). The continuous case (Theorems \ref{thmIntro1} and \ref{thmIntro2}) is discussed in Section \ref{sec:GF}. The proofs in that case are very similar to the discrete situation, so we only discuss necessary modifications.

{Several open problems and suggested in Section \ref{sec:open}.
It is worth mentioning that the groupoid approach above may also allow one to give a geometric description 
for yet another equivalent point of view on generalized flows,  taken by Brenier \cite{brenier0} and Shnirelman \cite{Sh2} (see also \cite[Section IV.7]{AK}), as 
probabilistic measures on the space of parametrized continuous paths in the flow domain.}
It would be also interesting to describe the group and Hamiltonian picture for vector and matrix densities in \cite{BV, ciosmak2021optimal} and 
the surprising appearance of the general relativity equation for matrix measures  in \cite{brenier2022interpretation}.
\medskip

{\bf Acknowledgements.} 
We are indebted to the MFO Institute in Oberwolfach, Germany and its program of Research in Pairs, where this work was completed. { We are also grateful to the anonymous referee for various suggestions  improving the exposition.}
A.I. was supported by NSF grant DMS-2008021. B.K. was partially supported by an NSERC Discovery grant.

\medskip

\section[Lie groupoids and algebroids]{Lie groupoids and algebroids}\label{sec:gen}
In this section we briefly recall basic facts about Lie groupoids and algebroids (details can be found, e.g., in {\cite{dufour2006poisson, mackenzie2005general}}). %

\subsection{Lie groupoids}\label{sect:groupoid}

\begin{definition}
A \textit{groupoid} $\G \rightrightarrows \B$ is a pair of sets, $\B$ (the \textit{base} of the groupoid) and $\G$ (the groupoid itself), endowed with the following structures:
\begin{enumerate} \item Two maps $\Src,\Trg \colon \G \to \B$, called \textit{source} and \textit{target} respectively.
\item Partial binary operation $(g, h) \mapsto gh$ on $\G$ which is defined for all pairs $g,h \in \G$ such that $\Src(g) = \Trg(h)$ and has the following properties:
\begin{enumerate}
\item  The source of the product is the source of the right factor: $
\Src(gh) = \Src(h)$, while the target of the product is the target of the left factor: $\Trg(gh) = \Trg(g).$
\item Associativity: $g(hk) = (gh)k$ whenever any of those expressions is well-defined.
\item Identity: for any $x \in \B$, there exists an element $\id_x \in \G$ such that $\id_{\Trg(g)} \cdot  g =  g\cdot \id_{\Src(g)} = g$  for every $g \in \G$.
\item Inverse: for any $g \in \G$, there exists an element $g^{-1} \in \G$ such that 
$g^{-1}g = \id_{\Src(g)}$ and $gg^{-1} = \id_{\Trg(g)}.$
\end{enumerate}
\end{enumerate}
A groupoid $\G \rightrightarrows \B$ is called a \textit{Lie groupoid} if $\G, \B$ are manifolds, the source and target are submersions, and the maps $(g, h) \mapsto gh$, $x \mapsto \id_x$, and $g \mapsto g^{-1}$ are 
smooth. (The domain of the multiplication map is $\{(x,y) \in \G \times \G \mid \Src(x) = \Trg(y)\}$. The submersion property of source and target ensures that this set is a submanifold of  $\G \times \G$, so smoothness of multiplication is well-defined.)
\end{definition}
\begin{example}\label{groupoids}
\begin{enumerate}[label=(\alph*)] 
\item Any Lie group $G$ is a Lie groupoid over a point.
\item For any smooth manifold $\B$, the set $\G := \B \times \B$ is a Lie groupoid over $\B$, called \textit{the pair groupoid}. The source and target are defined by $\Src(x,y) = x$, $\Trg(x,y) = y$, while the product is given by
$
(y,z)(x,y) := (x,z).
$

\item
Let $\B$ be a smooth manifold, and let $G$ be a Lie group acting on $\B$. Then the \textit{action} Lie groupoid $G \ltimes \B  \rightrightarrows \B$ is defined as follows. The points of $G \ltimes \B$ are triples $(g; x,y)$, where $x,y \in \B$, $g \in G$, and $gx= y$. The source map is given by $\Src(g; x,y) := x$, the target is $\Trg(g; x,y) := y$, and the multiplication is defined by
$
(h; y,z)(g; x,y) := (hg;x,z)\,.
$ 

\end{enumerate}
\end{example}
\begin{definition}\label{def:trans}
A groupoid $\G \rightrightarrows \B$ is called \textit{transitive} if for any $x, y\in \B$ there exists $g \in \G$ such that $\Src(g) = x$ and $\Trg(g) = y$. %For applications in this paper we will consider only transitive groupoids. %
\end{definition}
For example, an action groupoid  $G \ltimes \B$ is transitive if and only if the $G$-action on $\B$ is transitive. 

\begin{definition}
Let $\G \rightrightarrows \B$ be a groupoid.  Then the \textit{source fiber}  $\G_x$ of $\G$ corresponding to $x \in \B$ is the set  $\G_x := \{ g \in \G \mid \Src(g) = x\}$.
\end{definition}
For instance, for an action groupoid $G \ltimes \B$, any source fiber is canonically identified with the group $G$.%

\medskip

\subsection{Lie algebroids}\label{sect:algebroid}
The infinitesimal object corresponding to a Lie groupoid is a \textit{Lie algebroid}.
\begin{definition}
A \textit{Lie algebroid} $\A$ over a manifold $\B$ is a vector bundle $\A \to \B$ endowed with a Lie bracket $\LieBracket$ on smooth sections and a vector bundle map $\# \colon \A \to \T \B$, called the \textit{anchor}, such that for any two smooth sections $\zeta, \eta$ of $\A$ and any smooth function $f \in \Cont^\infty(\B)$, one has
$
[\zeta,f\eta] = f[\zeta,\eta] + (\L_{\#\zeta}   f) \eta\,.
$
\end{definition}
 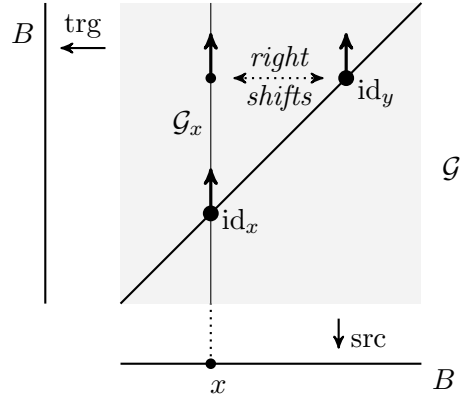
\begin{figure}[t]
\centerline{
\begin{tikzpicture}[thick, scale = 2]
     \fill [opacity = 0.05] (0,0) -- (0,2) -- (2,2) -- (2,0) -- cycle;
     \draw (2,-0.4) -- (0,-0.4);
     \draw (-0.5,0) -- (-0.5,2);
       \draw (0,0) -- (2,2);
            \fill (0.4+0.2,0.4+0.2) circle [radius=1.5pt];
 \draw  [thin] (0.4+0.2,0) -- (0.4+0.2,2);         
     \draw  [dotted] (0.4+0.2,-0.4) -- (0.4+0.2,-0);    
            \node  at (0.6+0.2,0.35+0.2) () {$\id_x$};
             \fill (1.3+0.2,1.3+0.2) circle [radius=1.5pt];
            \node  at (1.5+0.2,1.2+0.2) () {$\id_y$};
             \fill (0.4+0.2,1.3+0.2) circle [radius=1pt];
                 \fill (0.4+0.2,-0.4) circle [radius=1pt];
                      \node  at (0.45+0.2,-0.55) () {$x$};
                         \node  at (0.3+0.15,1.2) () {$\G_x$};
                             \node  at (2.2,0.9) () {$\G$};
 \draw  [dotted, <->] (0.55+0.2,1.3+0.2) -- (1.15+0.2,1.3+0.2);
            \node  at (0.85+0.2,1.4+0.23) () {\it right};
            \node  at (0.85+0.2,1.2+0.17) () {\it shifts};
 \draw  [very thick, ->] (0.4+0.2,0.4+0.2) -- (0.4+0.2,0.7+0.2);
 \draw  [very thick, ->] (0.4+0.2,1.3+0.2) -- (0.4+0.2,1.6+0.2);
 \draw  [very thick, ->] (1.3+0.2,1.3+0.2) -- (1.3+0.2,1.6+0.2);
  \draw  [->] (1.45,-0.1) -- (1.45,-0.3);
            \node  at (1.65,-0.25) () {$\rm src$};
    \draw  [->] (-0.1,1.7) -- (-0.4, 1.7);
            \node  at (-0.25,1.85) () {$\rm trg$};
            \node  at (2.15,-0.5) () {$B$};
            \node  at (-0.65,1.8) () {$B$};
\end{tikzpicture}
}
\caption{A groupoid $\G \rightrightarrows \B$. The
vertical projection is the source map $\Src \colon \G \to B$, the  horizontal
projection is the target map  $\Trg \colon \G \to B$, while horizontal arrows are right translations. A section of the algebroid is a
collection of vertical vectors attached to the diagonal $\Src = \Trg$.
}
\label{fig:algebroid}
\end{figure}

{The Lie algebroid $\A \to \B$ corresponding to a Lie groupoid $\G \rightrightarrows \B$} is constructed as follows.
The fiber of $\A$ over $x \in B$ is the tangent space to the source fiber $\G_x$ at the identity $\id_x$. The 
{ anchor map} on that fiber is defined as the differential of the target map $\Trg \colon \G_x \to \B$, while the 
{ bracket on sections} is defined as follows. Every section of $\A$ can be uniquely extended to 
a right-invariant vector field on $\G$ tangent to source fibers, and the correspondence between such vector 
fields and sections of $\A$ is a vector space isomorphism (see Figure \ref{fig:algebroid}). This allows one to define the bracket of sections 
of $\A$ as the Lie bracket of the corresponding right-invariant vector fields (which is again  
a right-invariant vector field tangent to source fibers, and, therefore, corresponds to a section of $\A$).

\begin{example}\label{algebroids}For Lie groupoids of Example \ref{groupoids}, the corresponding algebroids are:
\begin{enumerate}[label=(\alph*)] 
\item
The Lie algebra $\g$ of the group $G$, considered as a Lie algebroid over a point. { The anchor map is trivial, while the bracket on sections (which are simply elements of $\g$) is just the bracket on $\g$.}
\item The tangent bundle $\T B$ of  $B$. The corresponding bracket on sections is the standard Lie bracket of vector fields, while the anchor map is the identity.
\item
The \textit{action Lie algebroid} $\g \ltimes B$, where $\g$ is the Lie algebra of the group $G$. As a vector bundle, $\g \ltimes B$ is a trivial bundle over $B$ with fiber $\g$. The anchor map $\g \ltimes B \to \T B$ is defined for an element $(u,x) \in \g \ltimes B$  by $\#(u, x) = \rho_u({x}) $, where $\rho_u$ is the infinitesimal generator of the $G$-action corresponding to $u \in \g$. The bracket of sections is given by \begin{equation}\label{actionBracket}
[\zeta,\eta](x) := [\zeta(x), \eta(x)]_\g + (\L_{\#\zeta}\eta)(x)- (\L_{\#\eta}\zeta)(x)\,,
\end{equation}
where $\LieBracket_\g$ is the bracket in $\g$, and the derivatives $\L_{\#\zeta}\eta$, $\L_{\#\eta}\zeta$ are defined by identifying sections of $\g \ltimes B $ with $\g$-valued functions on $B$.\footnote{{It is useful to compare this bracket to that of a semidirect product Lie algebra $\tilde \g:=\g\ltimes B$, where the group of the Lie algebra $\g$ acts 
on a vector space $B$ (e.g. the Lie algebra for the group of affine transformations of $B$, the semidirect product of linear transformations and translations). The Lie bracket of $\tilde \g$ between two elements 
$(u,\alpha), (v,\beta)\in \g \ltimes B $ is $([u,v]_\g,   {\rm ad}_u\beta -{\rm ad}_v\alpha)\,.$} }
\end{enumerate}
\end{example}
\begin{definition}
A {\it Lie algebroid} $\A \to \B$ is called {\it transitive} if the anchor map is surjective.
\end{definition}
The Lie algebroid associated with a transitive Lie groupoid is transitive. 

\medskip

\subsection{Dual Lie algebroids as Poisson vector bundles}\label{sect:poisson_br}
Recall that the dual space $\g^*$ of any Lie algebra $\g$ carries a natural linear Poisson structure. This result extends to the algebroid setting: the dual of a Lie algebroid is a \textit{Poisson vector bundle}.

%
%
%
%
%
%This duality between Lie algebras and ``Poisson vector spaces'' extends to the vector bundles setting. The corresponding dual objects are {Lie algebroids} and \textit{Poisson vector bundles}.

\begin{definition}
A \textit{Poisson vector bundle} $ E \to B$ is a vector bundle whose total space $E$ is endowed with a fiberwise linear Poisson structure, i.e. a Poisson structure such that the bracket of any two fiberwise linear functions is again a fiberwise linear function.
\end{definition}
Two basic examples of Poisson vector bundles are a vector space endowed with a linear Poisson structure (which is a Poisson vector bundle over a point), and the cotangent bundle of a manifold $B$. These Poisson vector bundles are dual, respectively, to Lie algebroids $\g$ and $\T B$
from Examples \ref{algebroids}(a) and \ref{algebroids}(b). 
For general Lie algebroids, one has the following result.
\begin{proposition} 
The dual bundle\footnote{If the fibers of $\A$ are infinite-dimensional, then the fibers of $\A^*$ consist of sufficiently regular functionals on fibers of $\A$. In the hydrodynamical setting we will make this precise below.} $\A^* \to B$ of any Lie algebroid $\A \to B$ has a natural structure of a Poisson vector bundle. The Poisson structure on $\A^*$ is uniquely determined by requiring that for arbitrary fiberwise linear functions $\zeta ,\eta$ and an arbitrary fiberwise constant function~$f$, one has
$
\{ \zeta,\eta\} := [\zeta,\eta]$, $ \{\zeta, f\} :=\L_{\#\zeta}  f$. 
Here we identify fiberwise linear functions on $\A^*$ with sections of $\A$, and fiberwise constant functions on $\A^*$ with functions on the base~$B$.
\end{proposition}
In what follows, we will need the following explicit formula for the Poisson structure on  a Lie algebroid dual. 
%We first define the \textit{Lie algebroid differential}:
%\begin{definition}
%For an arbitrary $1$-form $\theta$ on $\A$ (i.e., a section of $\A^*$) its \textit{algebroid differential} $\diff_\A \theta$ is a $2$-form on $\A$ given for arbitrary $\zeta, \eta$ lying in one fiber of $\A$ by the formula
%\begin{align}\label{algDiff}
%\diff_{\A}\theta \left(\zeta,\eta\right) := -\theta([\hat \zeta,\hat \eta]) + \nabla_{\#\zeta}   (\theta(\hat \eta)) - \nabla_{\#\eta}   (\theta(\hat \zeta))\,,
%\end{align}
%where $\hat \zeta, \hat \eta$ are arbitrary smooth sections of $\A$ extending $\zeta, \eta$, and $\nabla_{\#\zeta}   (\theta(\hat \eta)) $ stands for the derivative of the function $ \theta(\hat \eta) \in \Cont^\infty(B)$ in the direction $\#\zeta$.
%\end{definition}
%
%For instance, for a tangent bundle $\A = TB$,  the operation $\diff_\A$ is the de Rham differential, while for a Lie algebra $\A = \g$ (considered as algebroid over a point) $\diff_\A$ is the Chevalley-Eilenberg differential.

%
%
%
%
%
%
%
%
%
%
%
%
%
%
%
%
%
%
%
%
%
%

\begin{proposition}\label{poissonBracketFormula} {\rm\cite{boucetta2011}} Let $\A$ be a Lie algebroid. Then,   for any $ \alpha \in \A^*$ 
and for any smooth functions $f, g \in \Cont^\infty(\A^*)$, one has
{{
%\begin{align}\label{PoissonExplicit}
%\{f, g\}(\alpha) =\langle  \alpha, [\diff ^Ff(\hat \alpha), \diff^Fg(\hat \alpha)]\vphantom{ \hat \alpha} \rangle + \L_{\#\diff^Ff(\alpha)}  \left(g \circ \hat \alpha - \langle\hat \alpha, d^Fg(\hat \alpha) \rangle\right)-  \dots,
%\end{align}
\begin{align}\label{PoissonExplicit}
\begin{gathered}
\{f, g\}(\alpha) =\langle  \alpha, [\diff ^Ff(\hat \alpha), \diff^Fg(\hat \alpha)]\vphantom{ \hat \alpha} \rangle 
+ \L_{\#\diff^Ff(\alpha)}  \left(g \circ \hat \alpha - \langle\hat \alpha, d^Fg(\hat \alpha) \rangle\right)\\
- \,\L_{\#\diff^Fg(\alpha)}  \left(f \circ \hat \alpha - \langle\hat \alpha, d^Ff(\hat \alpha) \rangle\right),
\end{gathered}
\end{align}
}where $\hat \alpha $ is an arbitrary section of $\A^*$ extending $\alpha$, 
 and $\diff^Ff(\alpha), \diff^Fg(\alpha) \in \A$ are fiber-wise differentials of $f$ and $g$ at $\alpha$ (i.e. differentials restricted to the tangent space of the fiber of $\alpha \in \A^*$).  %The dots denote a term similar to the previous one but with $f$ and $g$ interchanged. 
}
\end{proposition}

This formula can be used as a definition in the infinite-dimensional case. Although for general infinite-dimensional algebroids  it is not even clear why this expression makes sense, 
 we prove it below by obtaining an explicit formula in the setting of multiphase diffeomorphism groupoids.
 
 { Also note that for an action algebroid (see Example \ref{algebroids}c) formula \eqref{PoissonExplicit} becomes
 \begin{align}\label{PoissonExplicit2}
\{f, g\}(\alpha) =\langle  \alpha, [\diff ^Ff(\alpha), \diff^Fg(\alpha)]\vphantom{ \hat \alpha} \rangle + q(f,g) - q(g,f),
\end{align}
 where
 \begin{equation}\label{eq:qterm}
 q(f,g) := \langle \alpha, \L_{\#\diff^Ff(\alpha)}\diff^Fg(\hat \alpha) \rangle +  \L_{\#\diff^Ff(\alpha)}  \left(g \circ \hat \alpha - \langle\hat \alpha, d^Fg(\hat \alpha) \rangle\right).
 \end{equation}
 %for an arbitrary extension $\hat \alpha$ of $\alpha$.
 }
\medskip
\subsection[Euler-Arnold equations on Lie algebroids]{Euler-Arnold equations on Lie algebroids}\label{section:algsub}
Let $\A \to B$ be a  finite- or infinite-dimensional Lie algebroid, and let $\I \colon \A \to \A^*$ be an invertible 
bundle morphism. ({In the infinite-dimensional case one needs to consider the smooth dual bundle $\A^*$,
similarly to consideration of smooth duals of infinite-dimensional Lie algebras, cf.~\cite{AK}. In the hydrodynamical 
setting we define this smooth dual in detail in Section \ref{sec:kinematics}.})
We call such $\I$ an \textit{inertia operator}. An inertia operator $\I$ defines a metric on $\A$ given by
$$
 \langle u,v\rangle_\A := \langle \I(u), v \rangle \,
$$
for any $u ,v$ in the same fiber of $\A$.
Since the inertia operator $\I$ is invertible, one also has a dual metric on $\A^*$:
$$
\langle \alpha, \beta \rangle_{\A^*} :=   \langle \I^{-1}(\alpha), \beta \rangle = \langle  \I^{-1}(\alpha), \I^{-1}(\beta) \rangle_\A
$$
for any $\alpha, \beta$ in the same fiber of $\A^*$. 
Consider also a function $\H \in \Cont^\infty( \A^* )$ defined by
\begin{equation}\label{Hamiltonian}
\H(\alpha) := \frac{1}{2}\langle \alpha, \alpha \rangle_{\A^*}\quad \forall \,\alpha \in \A^*.
\end{equation}

\begin{definition}
The Hamiltonian equation associated with the Poisson structure on $\A^*$ and
the function $\H$ is called the \textit{groupoid Euler-Arnold equation} corresponding to the metric~$\metric_\A$.
\end{definition}
\begin{example}
When $\A$ is a Lie algebra, we obtain the standard notion of an Euler-Arnold equation on a Lie algebra dual. When $\A = \T B$ is the tangent bundle of a manifold $B$, the Euler-Arnold equation is the geodesic equation on $B$. { (The latter is, of course, a second order equation on $B$, but it becomes first order -- specifically, the algebroid Euler-Arnold equation -- if we interpret it as an equation on $TB$.)} 
\end{example}
\begin{remark}\label{rem:idofmetrics}
In the case when the algebroid $\A$ is associated with a certain Lie groupoid $\G$, solutions of the Euler-Arnold equation can be interpreted as geodesics of a right-invariant source-wise (i.e. defined only for vectors tangent to source fibers) metric on $\G$. In the transitive case those solutions can also be thought of as geodesics on any source fiber $\G_x$.
\end{remark}
Furthermore, an Euler-Arnold equation on a transitive algebroid $\A \to \B$ always gives rise to a certain geodesic flow on the base $\B$. Indeed, let $\A \to \B$ be a Lie algebroid.
Then, since the anchor map $\# \colon A \to \T \B$ is an algebroid morphism, the dual map $\#^* \colon \T^*\B \to \A^*$ is Poisson. %({In the infinite-dimensional 
%case, one needs to define $\T^*\B$ in such a way that its image under $\#^*$ belongs to the regular 
% dual $\A^*$.}) 
Note that if, moreover, the algebroid $\A$ is transitive, then $\#^*(\T^* \B)$ is a symplectic leaf 
in $\A^*$. Indeed, if $\A$ is transitive, then  the Poisson map $\#^*$ is injective, while  the image of 
a closed injective Poisson map of a symplectic manifold is always a symplectic leaf. 

\begin{proposition}\label{prop:sub} 
Let $\A \to \B$ be a transitive Lie algebroid, and let $\metric_\A$ be a positive-definite metric on $\A$ 
for an invertible inertia operator $\I \colon \A \to \A^*$. Assume also\footnote{Note that this property is automatic in the finite-dimensional case.} that for this metric $\metric_\A$  there  is an orthogonal decomposition 
$\A = \Ker \# \oplus (\Ker \#)^\bot$.  Then the following holds:
\begin{enumerate}
\item The pullback of the groupoid Euler-Arnold flow corresponding to the metric $\metric_\A$ from the symplectic leaf $\#^*(\T^*\B)$ to $\T^*\B$ is the geodesic flow for a certain metric $\metric_\B$ on $\B$. 
Explicitly, for any $x \in \B$ and any $\zeta, \eta \in \T_x\B$, the metric $\metric_B$ reads \begin{align}\label{inducedMetric}
\langle \zeta, \eta \rangle_\B = \langle \#^{-1}(\zeta),  \#^{-1}(\eta)\rangle_\A\,,
\end{align}
where $\#^{-1} \colon \T\B \to (\Ker \#)^\bot$ is the inverse for the restriction of the anchor map to $(\Ker \#)^\bot$.

\item  Assume, in addition, that the algebroid $\A$ corresponds to a certain transitive groupoid $\G$. Then, for every $x \in \B$, the target mapping $\Trg \colon (\G_x, \metric_\G) \to (\B, \metric_\B)$ is a Riemannian submersion. (Here the metric $\metric_\G$ on $\G_x$ is defined using the identification between metrics on $\A$ and right-invariant source-wise metrics on $\G$, see Remark \ref{rem:idofmetrics}.)
\end{enumerate}
\end{proposition}

For the proof see \cite{IK}.

\begin{example}\label{groupoidOverDensities} 
Let $M$ be a Riemannian manifold. Consider the natural transitive action of its diffeomorphism group $\Diff(M)$ on the space $\W(M)$ of densities on $M$ of unit total mass, and let $\Diff(M) \ltimes \W(M)$ be the corresponding action groupoid (see Example \ref{groupoids}(c)). Define a metric on the corresponding action algebroid $\vect(M) \ltimes \W(M)$ by setting
$$
\langle u, v \rangle_{L^2} := \int_M (u,v)\mu
$$
for $u, v$ lying in the fiber of $\vect(M) \ltimes \W(M)$ over $\mu \in \W(M)$. (Recall that the fibers of $\vect(M) \ltimes \W(M)$ are identified with the Lie algebra $\vect(M)$, see Example \ref{algebroids}(c).) Then,  according to Remark \ref{rem:idofmetrics}, for any $\mu \in \W(M)$, there is a corresponding metric on the source fiber $(\Diff(M) \ltimes \W(M))_\mu = \Diff(M)$. %invariant under the right action of the  group $\SDiff(M)=\{ \phi \in \Diff(M) \mid \phi^*\mu = \mu\}$. 
It is an $L^2$-type metric on $\Diff(M)$, and there is a Riemannian submersion of that metric onto Wasserstein metric on $\W(M)$, see~\cite{Otto} and  Remark \ref{rem:Was} below.  
\end{example}

\begin{example} 
Another example is given by the metric on the space $\VS(M)$ of vortex sheets in a manifold $M$, discussed in \cite{Loesch, IK} and in  Appendix below. 
In that case one considers  the Lie  groupoid $\DSDiff(M)$ of volume-preserving diffeomorphisms of a manifold $M$ that are discontinuous along a hypersurface. Its Lie algebroid $\dsvect(M) \to \VS(M)$ consists of  velocities  of the fluid with a vortex sheet: given a vortex sheet $\Sheet$, the corresponding velocities are discontinuous vector fields on $M$ of the form $u = \chiplus u^+ + \chimin u^-\!,$ where $\chiplus, \chimin$ are the indicator functions of connected components $\Dompm$ of $M\, \setminus \,\Sheet$, and $u^\pm$ are smooth divergence-free vector fields on $\Dompm$ such that the restrictions of $u^+$ and $u^-$ to $\Sheet$ have the same normal component, see appendix. %
There is a Riemannian submersion from an $L^2$ metric on $\DSDiff(M)$   to a metric on the space of classical vortex sheets, cf. \cite{Loesch, IK}.
\end{example}

\medskip

\section{Kinematics of multiphase fluids}\label{sec:kinematics}
In this section, $M$ is a compact connected manifold without boundary endowed with a volume form $\vol$. %

\subsection{The Lie groupoid of multiphase diffeomorphisms}\label{sect:diffeo_pair}
In this subsection, we define the Lie groupoid $\mdiff(M)$ of volume-preserving multiphase diffeomorphisms.
This groupoid (or, more precisely, any of its source fibers) can be viewed as the configuration space of a fluid with several phases penetrating through each other. { The conditions defining the groupoid $\mdiff(M)$ can be seen as integration of the corresponding infinitesimal equation \eqref{eq:sde}, just like the group of volume-preserving diffeomorphisms arises from ``integrating'' the divergence-free condition on the corresponding velocity field $u$.}
\par
The base of the groupoid $\mdiff(M)$ is, by definition, the space $\mdens(M)$ of  multiphase densities, i.e. the space of $n$-tuples
$ \bmu:=(\mu_1,...,\mu_n)$, where each $\mu_i \in \Dens(M)$ is a density (top-degree form) on $M$, satisfying the following conditions: 
\begin{enumerate} 
\item $\sum\nolimits_{j=1}^n\mu_j=\vol$. \item $\mu_j>0$ for all $j=1,...,n$ everywhere on $M$. 
\item $\int_M \mu_j = c_j$ for fixed constants $c_1, \dots, c_n \in \R$ {(such that $\sum_{j=1}^n c_j = \int_M \vol$)}.  \end{enumerate}
These densities can be thought of as densities of different
fractions of the fluid, that can penetrate through each other without resistance, subject only to the total incompressibility condition.
The case of two densities, $n=2$, supported on two different sides of the separating hypersurface $\Gamma\subset M$ 
corresponds to the regular vortex sheet $\Gamma$ for an incompressible flow in $M$. This case belongs to the closure of our space of positive densities. The general case with densities $(\mu_1,...,\mu_n)$
corresponds to a multiphase fluid where different phases can percolate through each other. 

The elements of  $\mdiff(M)$ are $n$-tuples of diffeomorphisms of $M$ preserving the above property of incompressibility of multiphase densities,
i.e. the set of tuples $(\bphi\,;\bmu, \bmu'):=(\phi_1, ..., \phi_n;\mu_1, ..., \mu_n, \mu'_1, ..., \mu'_n)$ with multiphase forms $\bmu, \bmu'\in \mdens(M)$ such that the 
multiphase diffeomorphism $\bphi$ push-forwards one of them to the other, $\bphi_*\bmu=\bmu'$ component-wisely. 
The source and the target of $ (\bphi\,;\bmu, \bmu')$ are, by definition, $\bmu$ 
and $\bmu'$ respectively. The multiplication in $\mdiff(M)$ is given by composition of 
 diffeomorphisms:
$$
(\bpsi\,;\bmu', \bmu'')(\bphi\,;\bmu, \bmu') := (\bpsi\bphi\,;\bmu, \bmu'')\,.
$$

 \begin{remark}\label{frechet}
 The groupoid	$\mdiff(M) \rightrightarrows \mdens(M)$  is a transitive Lie-Fr\'echet groupoid. The proof of the Lie-Fr\'echet property is the standard consideration similar to that for groups of diffeomorphisms, cf. \cite{IK}. One can consider a more general case of densities $\mu_j\ge 0$ on $M$, in which case the groupoid is not necessarily transitive. The latter case covers that of the usual vortex sheets $\Gamma\subset M$, corresponding to the indicator densities supported on two different sides of $M\setminus \Gamma$, see Appendix.
 
 Since $\mdiff(M)$ is a Lie-Fr\'echet groupoid, it follows that the corresponding algebroid is well-defined as a Fr\'echet vector bundle over $\mdens(M)$ with a bracket and anchor on smooth sections. We describe this algebroid in detail in the next section.
 \end{remark}

\begin{remark}One can also consider the groupoid $\mdiff(M)$ in the category of Hilbert manifolds modeled on Sobolev $H^s$ 
spaces for sufficiently large $s$, $s > \dim M /2 + 1$, similarly to, e.g., \cite{EbMars70} or~\cite[Remark 3.3]{KLMP}. 
Note also that if $n=1$, i.e. we have a one-phase fluid, with the groupoid $\mdiff(M)$ becoming the group $\SDiff(M)$
of $\mu$-preserving diffeomorphisms of $M$.
\end{remark}

 \begin{remark}\label{subgroupoid}
 Note that if we {drop} the requirement that the densities $\mu_j$ sum to $\vol$ in the definition of the groupoid $\mdiff(M)$,  we obtain the definition 
 of the action groupoid $ \Diff(M)^n \ltimes {\rm Dens}(M)^n$, which is the product of $n$ copies of the action groupoid $ \Diff(M) \ltimes {\rm Dens}(M)$
 (see Example~\ref{groupoids}(c)) corresponding to the natural 
 action of the group $\Diff(M)$ on the space of densities ${\rm Dens}(M)$. So, the action groupoid $ \Diff(M)^n \ltimes {\rm Dens}(M)^n$
 comes with a subgroupoid  $\mdiff(M)$. The 
 subgroupoid $\mdiff(M)$ inherits certain properties of the ambient action groupoid. 
 In particular, the brackets in the algebroids corresponding to these groupoids are given by the same formulas.
 \end{remark}

 \medskip

 \subsection{The Lie algebroid of multiphase vector fields}\label{sect:algebroid_vf}
 In this subsection we describe the Lie algebroid
$\mvect(M) \to \mdens(M)$
corresponding to the Lie groupoid $\mdiff(M)$. This algebroid serves as the space of velocities for a multiphase fluid.

\begin{theorem}\label{algFibers}
The Lie algebroid $ \mvect(M) \to  \mdens(M)$ 
corresponding to the groupoid $\mdiff(M)$ is as follows:
\begin{enumerate}\item
The fiber of $\mvect(M)$ over $\bmu:=(\mu_1,...,\mu_n) \in \mdens(M)$ is the space $\mvect(M, \bmu)$ which consists
 of multiphase vector fields on $M$ of the form 
$\bu := (u_1,...,u_n)$, where 
 $u_{j}\in \vect(M)$ are such that \begin{equation}\label{divFree}\sum\nolimits_{j=1}^n \L_{u_j}\mu_j=0\end{equation} (in other words, the multiphase vector field is ``divergence-free" with respect to the multiphase volume form). 
 
\item  The anchor map  $
 \# \colon \mvect(M, \bmu) \to \T_\bmu \mdens(M)\, $ is given by the negative Lie derivative, 
 $$
 \bu:=(u_1,...,u_n)\mapsto -\L_\bu\bmu:=
 (-\L_{u_1}\mu_1,..., -\L_{u_n}\mu_n)\,.
 $$
 
 \item Let $U, V$ be sections of $\mvect(M)$.
 Then their algebroid bracket is
 \begin{equation}\label{sectionsBracket}
[U,V](\bmu) = [U(\bmu), V(\bmu)] + {\L_{\#U(\bmu)} V - \L_{\#V(\bmu)} U\,,}
\end{equation}
where the first summand in the right-hand side is the usual Lie bracket of vector fields on $M$.
 \end{enumerate}
\end{theorem}

\begin{remark}\label{connectionInExtBundle} The derivative {$\L_{\#U(\bmu)} V$} is a multiphase vector field  defined by
$$
{\L_{\#U(\bmu)} V} :=\frac{ \hphantom{a}d}{d t}\restrict{t=0}{ V(\bmu(t))}\,,
$$
where $\bmu(t)$ is any smooth curve in $\mdens(M)$ with $\bmu(0) = \bmu$ and the tangent vector at $\bmu$ 
given by ${\#U(\bmu)}$. 
That derivative does not have to lie in $\mvect(M, \bmu)$, but belongs to the bigger
space  $\Vect(M)^n  := \{  (u_1,...,u_n)  \mid u_j\in \Vect(M))\}$
of $n$-tuples of vector fields on $M$ with no condition \eqref{divFree}. %
\end{remark}
\begin{remark}\label{remark:compensation}
Note that the first term on the right-hand side of \eqref{sectionsBracket} is not an element of $\mvect(M, \bmu)$. Indeed, for two multiphase vector fields $\bu$ and $\bv$ satisfying $\sum\nolimits_{i}\L_{u_j}\mu_j=\sum\nolimits_{i}\L_{v_j}\mu_j=0$, their (component-wise) Lie bracket does not necessarily have this property. However, the last two terms do not have this property either (see Remark \ref{connectionInExtBundle}) and  compensate the first term.
  \end{remark}

\begin{remark}\label{subAlgBracket}
The Lie algebroid $\mvect(M)$ is a subalgebroid in the action algebroid $ \Vect(M)^n \ltimes {\rm Dens}(M)^n$ of smooth multiphase vector fields without restrictions acting on volume multiphase forms  (see Remark \ref{subgroupoid}). Because of that  the  bracket in  $\mvect(M)$ automatically has form~\eqref{sectionsBracket} (cf. Example \ref{algebroids}(c)). %
However,  $\mvect(M)$ is \textit{not} an action algebroid. In particular, the fibers of  $\mvect(M)$ are not closed under the Lie bracket of vector fields (see Remark \ref{remark:compensation}) and hence do not have any natural Lie algebra structure.
\end{remark}

\proof[Proof of Theorem~\ref{algFibers}]
We begin with the first statement. By definition, the fiber  of $\mvect(M)$ over $\bmu$ consists of tangent vectors at $\units_\bmu \in \mdiff(M)$ to curves of the form $(\bphi(t); \bmu, \bmu(t) )$, where $\bmu(0) = \bmu$ and $\bphi(0) = (\id, \dots, \id)$. 
The  tangent vector to such a curve is an $n$-tuple of vector fields
$$
u_j =\frac{ \hphantom{}d}{d t}\restrict{t = 0}\phi_j(t) \,\,\in \,\Vect(M)\,.
$$
Also note that
\begin{equation}\label{boundaryRelation}
\sum\nolimits_{j=1}^n\phi_j(t)_* \mu_j = \sum\nolimits_{j=1}^n\mu_j(t) = \vol.
\end{equation}
Differentiating this relation with respect to $t$ at $t = 0$, we get \eqref{divFree}, as needed.
\par
Conversely, given any $n$-tuple of vector fields $u_j \in \Vect(M)$ satisfying \eqref{divFree}, one can construct a curve $\bphi(t)$ in the source fiber $\mdiff(M)_\bmu$ whose tangent vector at $\units_\bmu$ coincides with $\bu$. So, the fiber of $\mvect(M)$ over $\bmu$ is indeed the space $\mvect(M, \bmu)$.

To prove the second statement we need the following. 

\begin{lemma}\label{lemma:vstanspace}  The tangent space  $\T_\bmu \mdens(M)$ is the space of multiphase top-degree forms $\bxi=(\xi_1,..., \xi_n)$ satisfying the following conditions: \begin{enumerate} \item  $\sum\nolimits_{j=1}^n\xi_j=0$ on $M$.
\item $\int_M \xi_j = 0$ for each $i = 1, \dots, n$.
 \end{enumerate}
 \end{lemma}
\begin{proof}
These are infinitesimal versions of the conditions $\sum\nolimits_{i}\mu_j = \vol$ and $\int_M \mu_j = c_j$ respectively.
\end{proof}

 Now, we compute the anchor map (one can also get the formula for the anchor map using that our algebroid is a subalgebroid in the action algebroid, see Remark \ref{subAlgBracket}). Let $\bu\in \mvect(M)_\bmu$. Consider a curve $ (\bphi(t) ; \bmu, \bmu(t) ) \in \mdiff(M)_\bmu$ where 
 $\bmu(t) := \bphi(t)_*\bmu$ and
 whose tangent vector at $\units_\bmu$ is $\bu$. Then, by definition of the anchor map for the algebroid of a Lie groupoid, we have 
 $$
 \#\bu = \frac{ \hphantom{}d}{d t}\restrict{t=0}\Trg(\bphi(t)) =\frac{ \hphantom{}d}{d t}\restrict{t = 0} \bphi(t)_*\bmu = -\L_{\bu}\bmu \,, 
 $$
as desired.\par
Finally, the last statement of the theorem follows from the fact that our algebroid is a subalgebroid in the action algebroid (see Remark \ref{subAlgBracket}) and formula \eqref{actionBracket} for the action algebroid bracket. Thus, Theorem~\ref{algFibers} is proved.
\proofend

 \subsection{The dual algebroid and its tangent space}\label{sect:dual}
 In this subsection, we describe the dual of the Lie algebroid $\mvect(M)$. This space can be viewed as the space of momenta for a multiphase fluid.
 
 \par%
As the dual of $\mvect(M)$, we consider the ``smooth dual bundle'' defined as follows. 
 
 {
 \begin{definition}
 The smooth dual $\mvect(M,\bmu)^*$ of the space $\mvect(M,\bmu)$ is the space of linear functions $f \colon \mvect(M,\bmu) \to \R$ that admit a smooth density, which means that there exist smooth $1$-forms $\alpha_1, \dots, \alpha_n$ such that
\begin{equation}\label{eq:spairing}
 f(\bu) =  \sum\nolimits_{j=1}^n\int_M \alpha_j(u_j)  \mu_j
 \end{equation}
 for all $\bu  \in   \mvect(M,\bmu)$.
 \end{definition}
 Formula \eqref{eq:spairing} defines a surjective linear map $\pi \colon \Omega^1(M)^n \to \mvect(M,\bmu)^*$: each collection $(\alpha_1, \dots, \alpha_n) \in\Omega^1(M)^n $ is mapped to a linear function on $\mvect(M,\bmu)$ defined by \eqref{eq:spairing}.
 \begin{proposition}\label{prop:nnp}
 The kernel of the map  $\pi \colon \Omega^1(M)^n \to \mvect(M,\bmu)^*$ consists of $n$-tuples of the form $(df, \dots, df)$, where $f \in C^{\infty}(M)$. Therefore, we have an isomorphism $$
\mvect(M,\bmu)^*:= \Omega^1(M)^n/\delta(d\Cont^\infty(M)),
$$
where $ \delta \colon \Omega^1(M) \to  \Omega^1(M)^n$ is the diagonal embedding $\delta(\alpha) = (\alpha, \dots, \alpha)$.
 \end{proposition}
 \begin{proof}
 First observe that $\delta(d\Cont^\infty(M)) \subset \Ker \,\pi$. This is due to condition~\eqref{divFree}:  
\begin{gather}
\langle (df, \dots, df), \bu\rangle=\sum\nolimits_{j=1}^n\int_M df(u_j)  \mu_j
=-\sum\nolimits_{j=1}^n\int_M f\wedge \L_{u_j}\mu_j \\=- \int_M f\wedge \sum\nolimits_{j=1}^n\L_{u_j}\mu_j=0\,.
\end{gather}
So, the map $\pi  \colon  \Omega^1(M)^n \to \mvect(M,\bmu)^*$ descends to a surjective linear map $$\underline \pi \colon  \Omega^1(M)^n/\delta(d\Cont^\infty(M)) \to  \mvect(M,\bmu)^*.$$
We need to show that the latter is injectibe. This is based on the following.

\begin{proposition}\label{prop:ccRepr}
For any choice of a Riemannian metric on $M$, any coset $[\balpha] \in  \Omega^1(M)^n/\delta(d\Cont^\infty(M))$ has a unique (``co-closed'') representative $\balpha \in  \Omega^1(M)^n$ such that
\begin{equation}\label{eq:cc}d^*\sum\nolimits_{j=1}^n\rho_j \alpha_j = 0\,,\end{equation}
where $\rho_j : = \mu_j / \vol$.
\end{proposition}
\begin{proof}
This is equivalent to saying that for any $\balpha \in  \Omega^1(M)^n$ there exists a function $f \in \Cont^\infty(M)$, unique up to an additive constant, such that
$$
d^*\sum\nolimits_{j=1}^n\rho_j (\alpha_j + df) = 0 \quad \Leftrightarrow \quad d^* \left( df +  \sum\nolimits_{j=1}^n\rho_j \alpha_j  \right) = 0 \quad \Leftrightarrow \quad \Delta f = - d^* \sum\nolimits_{j=1}^n\rho_j \alpha_j\,.
$$
This is a Poisson equation on $f$ whose solution is unique up to an additive constant, as needed. 
\end{proof}
Returning to the proof of Proposition \ref{prop:nnp}, given $[\balpha] \in  \Omega^1(M)^n/\delta(d\Cont^\infty(M))$, $[\balpha] \neq 0$, consider its representative  $\balpha$ satisfying \eqref{eq:cc} for some Riemannian metric on $M$. Then the multiphase vector field $\balpha_{}^\sharp$  satisfies \eqref{divFree}, so $ \balpha^\sharp \in \mvect(M, \bmu)$. Furthermore, we have 
$$
\left \langle \underline \pi([\balpha]), \balpha_{}^\sharp \right\rangle  = \sum\nolimits_{j=1}^n\int_M \left(\alpha_j^\sharp, \alpha_j^\sharp\right)  \mu_j> 0\,,$$
and hence $\underline \pi([\balpha]) \neq 0$. So $\underline \pi$ is indeed injective, as needed. 
 \end{proof}
 In what follows we make an identification $\mvect(M,\bmu)^* \simeq  \Omega^1(M)^n/\delta(d\Cont^\infty(M))$.
%\begin{definition}
%The smooth dual $\mvect(M,\bmu)^*$ of the space $\mvect(M,\bmu)$  is the quotient
%$$
%\mvect(M,\bmu)^*:= \Omega^1(M)^n/d\Cont^\infty(M) %
%$$
%whose elements are of the form 
%$$
%[\balpha]:=\{(\alpha_1 +df,...,\alpha_n +df) \mid f\in \Cont^\infty(M) \}
%$$
%for some $\alpha_1, \dots, \alpha_n \in \Omega^1(M)$.
%(This is the diagonal embedding of $d\Cont^\infty(M)$ into $ \Omega^1(M)^n$: all one-forms in the $n$-tuple are defined modulo the same function differential.)
%The pairing between a coset $[\balpha] \in  \mvect(M, \bmu)^*$ and a multiphase vector field $\bu  \in   \mvect(M,\bmu)$ is given by the formula
%\begin{align}\label{cosetAction}
%\langle[\balpha], \bu\rangle := \sum\nolimits_{j=1}^n\int_M \alpha_j(u_j)  \mu_j\,, 
%\end{align}
%where $(\alpha_1,...,\alpha_n)$ is any representative of the coset $[\balpha]$. 
%%
Accordingly, the smooth dual of the algebroid $\mvect(M)$ is the trivial vector bundle
 $$
 \mvect(M)^* = \bigcup\nolimits_{\bmu \in \mdens(M)}\mvect(M,\bmu)^* = \left( \Omega^1(M)^n/\delta(d\Cont^\infty(M)) \right) \times \mdens(M).
 $$ 
% It is a trivial vector bundle
% $$
% \mvect(M)^*=\left( \Omega^1(M)^n/d\Cont^\infty(M) \right) \times \mdens(M)
% $$
 over the space of multiphase densities $\mdens(M)$.

 An important property of the smooth dual $ \mvect(M, \bmu)^* =  \Omega^1(M)^n/\delta(d\Cont^\infty(M))$ is that this subspace of the full dual space ``separates points'', meaning that for any non-zero $\bu \in \mvect(M, \bmu)$ there exists $[\balpha] \in \mvect(M, \bmu)^*$ such that $\langle [\balpha], \bu\rangle \neq 0$ (indeed, one can take $\balpha := \bu^\flat$). This is equivalent to saying that $\mvect(M, \bmu)$ injects into the dual of its smooth dual, which is needed for the Poisson bracket on $\mvect(M)^*$ to be well-defined. 
This Poisson bracket is described in the next section. For this we describe the tangent and cotangent spaces to $\mvect(M)^*$. First we note that since $\mvect(M)^*$ is a trivial vector bundle, we have the following natural splitting of its tangent space.
 }
 %\end{definition}
% 
%Note that at each fiber the value of $\langle[\balpha], \bu\rangle$ does not depend on the choice of the representative $(\alpha_1,...,\alpha_n) \in [\alpha]$ because for $\bff = (f, \dots, f)$, where $f \in \Cont^\infty(M)$, the pairing between $d\bff:=(df,...,df) $ and $\mvect(M, \bmu)$ vanishes due to condition \eqref{divFree}:  
%$$
%\langle d\bff, \bu\rangle=\sum\nolimits_{j=1}^n\int_M df(u_j)  \mu_j
%=-\sum\nolimits_{j=1}^n\int_M f\wedge \L_{u_j}\mu_j=- \int_M f\wedge \sum\nolimits_{j=1}^n\L_{u_j}\mu_j=0\,.
%$$
%So, any  coset $[\balpha] \in  \mvect(M, \bmu)^*$  gives rise to a well-defined linear functional on the space $\mvect(M, \bmu)$. For $[\balpha] \neq 0$ this functional is non-zero by the following proposition.
%
%
%\begin{corollary}\label{cor:nonVanish}
%Any non-zero coset $[\balpha] \in  \Omega^1(M)^n/d\Cont^\infty(M)$ defines a non-zero functional on the space $\mvect(M, \bmu)$.
%\end{corollary}
%\begin{proof}
%Indeed, let  $[\balpha] \in  \Omega^1(M)^n/d\Cont^\infty(M)$ be non-zero, and let $\balpha_{} \in [\balpha]$ be its representative satisfying \eqref{eq:cc} (for some Riemannian metric on $M$). Then the multiphase vector field $\balpha_{}^\sharp$  satisfies \eqref{divFree}, so $ \balpha^\sharp \in \mvect(M, \Sheet)$. Furthermore, we have 
%$$
%\left\langle[\balpha], \balpha^\sharp\right\rangle =  \sum\nolimits_{j=1}^n\int_M \left(\alpha_j^\sharp, \alpha_j^\sharp\right)  \mu_j> 0\,,$$
%so the functional defined by $[\balpha]$ is indeed non-zero.
%\end{proof}
%
%%
%
% %
%%
%
%
% \medskip

%
%
%
%
%
%
%
%
%
%
%
%
%
%
%
%
%
%
%
%
%
%
%
%
   \begin{proposition}
There is a natural splitting
 \begin{equation}\label{tanSpaceSplitting}
 \T_{[\balpha],\,\bmu}\mvect(M)^* \simeq \mvect(M, \bmu)^*_{}\, \oplus \, \T_\bmu\mdens(M)\,.
 \end{equation}
 \end{proposition}

  \subsection{Poisson bracket on the dual algebroid}\label{sect:poisson_vs}
 In this section we show that formula \eqref{PoissonExplicit} gives a well-defined Poisson bracket on $\mvect(M)^*$. For this we need to describe the cotangent space to  $\mvect(M)^*$ and we  start by defining the cotangent space to the base,
  $\T^*_\bmu \mdens(M)$.
 \begin{definition}
{The \textit{smooth cotangent space}  $\T^*_\bmu \mdens(M)$ is the quotient $\Cont^\infty_0(M)^n \, / \, \delta(\Cont^\infty_0(M))$, where $\Cont^\infty_0(M) := \Cont^\infty(M)\,/ \, \R$, and $\delta \colon \Cont^\infty_0(M) \to \Cont^\infty_0(M)^n $ is the diagonal embedding $h \mapsto (h, \dots, h)$. The pairing between a coset $[\bF] \in   \Cont^\infty_0(M)^n \, / \, \delta(\Cont^\infty_0(M)) $ }and a tangent vector $\bxi \in  \T_\bmu \mdens(M)$ (i.e. a collection $(\xi_1, \dots, \xi_n)$ of top-degree forms on $M$ such that $\sum\nolimits_{j=1}^n \xi_j = 0$, see Lemma \ref{lemma:vstanspace}) is given by 
  $$
  \langle [\bF], \bxi \rangle := \sum\nolimits_{j=1}^n \int_M f_j \xi_j\,.
  $$
  (The right-hand side {does not depend on} the choice of a representative $\bF \in [\bF]$ thanks to the zero sum condition on $\bxi$. The integral $\int_M f_j \xi_j$ is well-defined for $f_j \in \Cont^\infty_0(M)$ since $\int_M  \xi_j = 0$.)
  \end{definition}
  Now we define the cotangent space to $\mvect(M)^*$ by dualizing splitting \eqref{tanSpaceSplitting}.
  
   \begin{definition}
 Let ${([\balpha],\bmu)} \in \mvect(M)^*$. Then the \textit{smooth cotangent space} to $\mvect(M)^*$ at ${([\balpha],\bmu)}$ is
 \begin{equation}\label{cotanSpaceSplitting}
 \T^*_{{[\balpha],\,\bmu}}\mvect(M)^* := \mvect(M, \bmu)_{}\, \oplus \, \T^*_{\bmu}\mdens(M)\,,
\end{equation}
 where the second summand is the smooth cotangent space. %
 \end{definition}
 Further, we define the notion of a differentiable function on $\mvect(M)^*$. Roughly speaking, a function is differentiable if it has a differential belonging to the smooth cotangent space.
 \begin{definition}
 A function $\F \colon \mvect(M)^* \to \R$ is \textit{differentiable} if there exists a section $d\F$ of the smooth cotangent bundle  $\T^*_{}\mvect(M)^*$
such that for any smooth curve $\left([\balpha](t), \mu(t)\right)$ in $\mvect(M)^*$ one has
 $$
\tfrac{d}{dt}\F\left(\vphantom{}{[\balpha](t)}, \bmu(t)\right)= \langle  d \F\left({[\balpha](t)}, \bmu(t)\vphantom{}\right),({\tfrac{d}{dt} [\balpha](t)}{}, {\tfrac{d}{dt}\bmu(t)}{}  ) \rangle\,.
 $$
 \end{definition}
 Using splitting \eqref{cotanSpaceSplitting}, we decompose $d\F({[\balpha]})$ for $[\balpha] \in \mvect(M, \bmu)^*$ into the fiber and base parts: 
  $$
d \F\left({[\balpha]}, \bmu\vphantom{}\right) = \left(
 {d^F\F}\left({[\balpha]}, \bmu\vphantom{}\right),
  {d^B  \F}\left(\vphantom{}{[\balpha]}, \bmu\right)
  \right),$$
  where
  $$
  {d^F\F}\left({[\balpha]}, \bmu\vphantom{}\right)\in \mvect(M, \bmu)\,, \quad      {d^B \F}\left({[\balpha]}, \bmu\vphantom{}\right) \in \T^*_{\bmu}\mdens(M) = \Cont_0^\infty(M)^n \, / \, \delta(\Cont_0^\infty(M)).
 $$
 
 \begin{theorem}\label{thm:bracket}
 Let  $\F_1, \F_2 \colon \mvect(M)^* \to \R$ be differentiable functions. Then their Poisson bracket reads%
 \begin{equation}
   \{\F_1, \F_2\} = \P(d \F_1, d \F_2)\,,
 \end{equation}
 where the value of the Poisson tensor $\P$ on two cotangent vectors $$(\bu, [\bF]), (\bv, [\bg])\in \T^*_{[\balpha], \,\bmu}\mvect(M)^* = \mvect(M, \bmu)\, \oplus \, \T^*_{\bmu}\mdens(M)$$ at a point $([\balpha], \bmu) \in \mvect(M)^*$ is
 {
\begin{equation}\label{pbf}
\begin{aligned}
 \P_{[\balpha], \,\bmu}&\left(\vphantom{}(\bu, [\bF]), (\bv, [\bg]) \right)
  =
  \sum\nolimits_{j=1}^n\int_M \left( -\diff \alpha_j(u_j, v_j) + \L_{u_j}g_j - \L_{v_j}f_j\right)\mu_j.  %
  \end{aligned}
\end{equation}
}
Here $\balpha_{} \in \Omega^1(M)^n$ is  an arbitrary  representative of the coset $[\balpha]$.
 \end{theorem}
 {
\begin{remark}
 Equivalently, this bracket can be written in the form, similar to  a Lie-Poisson bracket with additional terms:
\begin{align}\label{pbf2}
 \begin{aligned}
 \P_{[\balpha], \, \bmu}\left(\vphantom{}(\bu, [\bF]), (\bv, [\bg]) \right)& \\ = \sum\nolimits_{j=1}^n&\int_M \left(\alpha_{i}([ u_j, v_j ])
 + \, \L_{v_j}\left( i_{u_j}\alpha_j-f_j \right) -  \, \L_{u_j}\left( i_{v_j}\alpha_j-g_j \right) \right)\mu_j.
\end{aligned}
\end{align}
\end{remark}

 \begin{proof}[Proof of Theorem \ref{thm:bracket}]
 {
 Formulas \eqref{pbf} and   \eqref{pbf2} are equivalent to each other. To see this, rewrite the first term in  \eqref{pbf2} using the formula $i_{[u, v]} = [\L_{u}, i_{v}]$. So, it suffices to derive~\eqref{pbf2}.
}

{ Since the bracket~\eqref{sectionsBracket} on sections of $\mvect(M)$ has the same form as for an action algebroid, we can compute the Poisson bracket in the dual using formula \eqref{PoissonExplicit2}. This gives
\begin{align}\label{PB1}
  \{\F_1, \F_2\}\left([\balpha], \bmu\vphantom{}\right) =    \langle \balpha ,\left[  d^F\F_{1}([\balpha] , \bmu),   d^F\F_{2}([\balpha] , \bmu)\right] \rangle\, + q (\F_1, \F_2) -  q (\F_2, \F_1),
\end{align}
Here and below  the pairing $\metric$ between multiphase forms and multiphase vector fields is given by {\eqref{eq:spairing}}, and the commutator of multiphase vector fields is defined component-wisely. To compute the $q(\F_k, \F_l)$ terms we extend $[\balpha]$ to a \emph{constant} section $A \colon \mdens(M) \to \mvect(M), A(\bmu) := [\balpha]$ of the trivial vector bundle $\mvect(M)^*$. Also, let $U_k$ be  the section of $\mvect(M)$ given by $U_k:=  {d^F\F_{k}}(A)$. Then formula \eqref{eq:qterm} gives
\begin{align}\label{sij}
q(\F_k, \F_l) := \langle\balpha,  \L_{\#U_k(\bmu)}  U_l \rangle +  \L_{\#U_k(\bmu)}  \left(\F_l \circ A\right) -\L_{\#U_k(\bmu)} \langle A ,U_l \rangle.
\end{align}

}
%  Formula \eqref{PoissonExplicit} for the Poisson bracket in the dual of an algebroid combined with formula~\eqref{sectionsBracket} for the bracket of sections of $\mvect(M)$  gives
%\begin{align}\label{PB1}
%  \{\F_1, \F_2\}\left([\balpha], \bmu\vphantom{}\right) =   S_{12} - S_{21} + \langle \balpha ,\left[  U_1(\bmu),   U_2(\bmu)\right] \rangle\,.
%\end{align}
%Here and below  the pairing $\metric$ between multiphase forms and multiphase vector fields is given by {\eqref{eq:spairing}}, the commutator of multiphase vector fields is defined component-wisely, while the expression $S_{kl}$ stands for
%\begin{align}\label{sij}
%S_{kl} :=  (\L_{\#U_k}  \left(\F_l \circ A\right))(\bmu) -(\L_{\#U_k} \langle A ,U_l \rangle) (\bmu)  +\langle\balpha,  \nabla_{\#U_k(\bmu)}  U_l \rangle\,,
%\end{align}
%where $A \colon \mdens(M) \to \mvect(M)$ is a constant section of the trivial vector bundle $\mvect(M)^*$ given by $A(\bmu) := [\balpha]$, and $U_k$ is a section of $\mvect(M)$ given by
%$
%U_k:=  {d^F\F_{k}}(A).
%$
\par
Now, take any curve $ \bmu_k(t) \in \mdens(M)$ such that $\bmu_k(0) = \bmu$, and the tangent vector to $\bmu_k(t)$ at $\bmu$  is $\#U_k(\bmu)$. %
Then, using that $A$ is a constant section, we get
\begin{align}\label{sijsum1}
\begin{gathered}
{\L_{\#U_k(\bmu)}  \left(\F_l \circ A\right) =\frac{ \hphantom{}d}{d t}\vert_{t = 0}\,  \F_l(A(\bmu_k(t)))}
  =  \langle \vphantom{\frac{ d}{d t}\restrict{t=0}  [\alpha_j] } {d^B  \F_l}{ }([\balpha], \bmu), \#U_k(\bmu) \rangle. % = {  \langle \vphantom{\frac{ d}{d t}\restrict{t=0}  [\alpha_j] } {d^B  \F_l}{ }([\balpha], \bmu), \# d^F\F_{k}([\balpha] , \bmu) \rangle}.
 \end{gathered}
\end{align}
Further, we have
\begin{gather}
{\L_{\#U_k(\bmu)} \langle A ,U_l \rangle}  = \frac{ \hphantom{}d}{d t}\restrict{t=0} \left\langle A(\bmu_k(t)) ,U_l(\bmu_k(t))\vphantom{} \right\rangle  = \frac{ \hphantom{}d}{d t}\restrict{t=0} \sum\nolimits_{j=1}^n \int_M \left(i_{U_l(\bmu_k(t))_j} \alpha_j\right) \mu_k(t)_j
 \vphantom{\frac{ \hphantom{}d}{d t}\restrict{t=0} \sum\nolimits_{j=1}^n \int_M \left(i_{U_l(\bmu_k(t))_j} \alpha_j\right) \mu_k(t)_j} \\ {=\sum\nolimits_{j=1}^n \int_M \left(i_{\frac{ \hphantom{}d}{d t}\restrict{t=0} U_l(\bmu_k(t))_j}\alpha_j\right) \mu_k(t)_j  + \sum\nolimits_{j=1}^n \int_M \left(i_{U_l(\bmu_k(t))_j}\alpha_j\right) \frac{ \hphantom{}d}{d t}\restrict{t=0} \mu_k(t)_j }
 % \\ {= \left\langle \balpha, {\L}_{\#U_k(\bmu)} U_l \vphantom{}\right\rangle - \sum\nolimits_{j=1}^n \int_M \left(i_{U_l(\bmu)_j}\alpha_j\right) \L_{U_k(\bmu)_j}\mu_j }
 \\=   \left\langle \balpha, {\L}_{\#U_k(\bmu)} U_l \vphantom{}\right\rangle  + \left\langle i_{U_l(\bmu)}\balpha,  \#U_k(\bmu)\vphantom{} \right\rangle\,.
\end{gather}
Substituting this, along with \eqref{sijsum1}, into \eqref{sij}, { we get
\begin{align}\label{sij2}
q(\F_k, \F_l) := \left\langle  {d^B  \F_l}{ }([\balpha], \bmu)-  i_{U_l(\bmu)}\balpha,  \#U_k(\bmu)\vphantom{} \right\rangle.
\end{align}
Finally, plugging this } into~\eqref{PB1} { and using that $U_l(\bmu) =  d^F\F_{l}([\balpha] , \bmu)$}, one gets  \eqref{pbf2}.  \end{proof}

 \begin{corollary}
 The  Hamiltonian operator
 $$
 \P_{[\balpha], \,\bmu}^\sharp \colon \T^*_{[\balpha], \,\bmu}\mvect(M)^* \to  \T_{[\balpha], \,\bmu}\mvect(M)^*
 $$
 corresponding to the Poisson bracket on $\mvect(M)^* $  is given by
 \begin{equation}\label{HamOperator}
  (\bu, [\bF]) \mapsto (-[i_{\bu}\diff \balpha] - [d\bF], -\L_\bu \bmu)\,.
\end{equation}
(Note that the coset $[d\bF]$ of $d\bF$ in $\Omega^1(M)^n / d\Cont^\infty(M)$ does not depend on the choice of a representative $\bF$ in the coset $[\bF] \in \Cont_0^\infty(M)^n / \Cont_0^\infty(M)$.)
 \end{corollary}
 \begin{proof}
By definition, we have
\begin{align}
 \langle (\bv, [\bg]) ,  \P_{[\balpha],\,\bmu}^\sharp(\bu, [\bF])\rangle &=  \P_{[\balpha],\,\bmu}\left((\bu, [\bF]), (\bv, [\bg])\right) \\&=    \sum\nolimits_{j=1}^n\int_M \left( -\diff \alpha_j(u_j, v_j) + \L_{u_j}g_j - \L_{v_j}f_j\right)\mu_j 
\\&=\vphantom{ \sum\nolimits_{j=1}^n\int_M \left( -\diff \alpha_j(u_j, v_j) + \L_{u_j}g_j - \L_{v_j}f_j\right)\mu_j } \langle -[i_{\bu}\diff \balpha] - [d\bF], \bv \rangle  - \langle \L_\bu \bmu, [\bg] \rangle.
 \end{align}
The result follows.
\end{proof}
\begin{remark}
Using Proposition \ref{dualAnchorProp} below, the Hamiltonian operator \eqref{HamOperator} can be rewritten in terms of the anchor map $\# \colon \mvect(M) \to T\mdens(M)$, as follows:
$$
(\bu, [\bF]) \mapsto (-[i_{\bu}\diff \balpha] - \#^*[\bF], \# \bu),
$$
which  agrees with the corresponding operator for the motion of vortex sheets, see \cite[Corollary 6.27]{IK}.
\end{remark}
 \medskip

\section{Dynamics of multiphase fluids}\label{sect:dyn_vs}
\subsection{Geodesic and Hamiltonian framework for multiphase fluids}

In this section, $M$ is a compact connected oriented manifold without boundary endowed with a Riemannian metric $(\,,)$ and the corresponding Riemannian volume form $\vol$.
We define a metric $\metric_{L^2}$ on the Lie algebroid $\mvect(M)$ as follows: for $\bu, \bv \in \mvect(M, \bmu) $, one has
\begin{equation}\label{eq:metric}
\langle \bu, \bv \rangle_{L^2} := \sum\nolimits_{j=1}^n \int_M (u_j, v_j) \mu_j.
\end{equation}
\begin{proposition}
\begin{enumerate}
\item
The inertia operator $\I$ associated with the $L^2$-metric $\metric_{L^2}$ on $\mvect(M)$  takes values in the smooth dual $\mvect(M)^*$. For $\bu \in \mvect(M, \bmu)$, one has
$
\I(\bu) = [\bu^\flat]
$,
where $$\bu^\flat := (u_1^\flat, \dots, u_n^\flat),$$ $u_j^\flat$ denotes the $1$-form dual to the vector field $u_j$ with respect to the Riemannian metric $(\,,)$ on $M$, and $[\bu^\flat]$ stands for the coset of $\bu^\flat$ in {$ \Omega^1(M)^n \, / \, \delta(d\Cont^\infty(M))$}.
\item The inertia operator $\I \colon \mvect(M) \to  \mvect(M)^*$ is an isomorphism of vector bundles.
\end{enumerate}
\end{proposition}\label{prop:invertibility}

\begin{proof}
By definition of the inertia operator, for $\bu ,\bv \in \mvect(M, \bmu)$, one has
$$
\langle \I(\bu), \bv \rangle = \langle \bu, \bv \rangle_{L^2} =  \sum\nolimits_{j=1}^n \int_M (u_j, v_j) \mu_j =  \sum\nolimits_{j=1}^n \int_M i_{v_j} u_j^\flat \,\mu_j\,.
$$
This means that the functional $\I(\bu)$ coincides with the functional represented by the coset of $\bu^\flat \in \Omega^1(M)^n$, proving the first statement.

The second statement, i.e. invertibility of the inertia operator, is equivalent to saying that the equation
$
[\bu^\flat] = [\balpha]
$
has a unique solution $\bu \in \mvect(M, \bmu)$ for any coset $[\balpha] \in \mvect(M, \bmu)^*$. Written in terms of the form $\balpha := \bu^\flat$, the condition  $\bu \in \mvect(M, \bmu)$ translates to \eqref{eq:cc}. { (Indeed, the defining condition \eqref{divFree} of $\mvect(M, \bmu)$ is equivalent to \eqref{eq:sde}, which, in turn, is equivalent to  \eqref{eq:cc} since the divergence of a vector field is the same as the co-differential of its metric dual form.) So,} the result follows from Proposition \ref{prop:ccRepr}. Explicitly, we have  $\I^{-1}([\balpha]) = \balpha^\sharp$, where $\balpha$ is the coset representative satisfying \eqref{eq:cc}.
\end{proof}
Since the inertia operator is invertible, we also obtain an $L^2$-metric on  $\mvect(M)^*$, and the corresponding Euler-Arnold Hamiltonian
$$
\mathcal H\left([\balpha], \bmu\vphantom{}\right) := \tfrac{1}{2} \sum\nolimits_{j=1}^n \int_M (\alpha_j, \alpha_j)\, \mu_j,
$$
where $\balpha \in [\balpha]$ is a representative satisfying \eqref{eq:cc}.

\begin{theorem}\label{thmMain}
The Euler-Arnold equation corresponding to the $L^2$-metric on $\mvect(M)$ written in terms of a coset $[\balpha] \in \mvect(M, \bmu)^*$ reads
\begin{subnumcases}{\label{twoPhaseEulerCosets}}
\partial_t [\balpha] + [i_\bu \diff \balpha + \tfrac{1}{2}\diff i_\bu\balpha]  =0\,,\label{twoPhaseEulerCosets1}\\
\qquad \partial_t \bmu = -\L_\bu \bmu\,, \label{twoPhaseEulerCosets2}
\end{subnumcases}
where $\balpha \in [\balpha]$ is the representative satisfying \eqref{eq:cc}, and $\bu := \balpha^\sharp \in \mvect(M, \bmu)$. 
It is a Hamiltonian equation on the algebroid dual $\mvect(M)^*$ with respect to the natural Poisson structure described above and the energy Hamiltonian function~$\mathcal H$.
\end{theorem}

\begin{remark}

Note that for a single phase fluid ($n=1$), the equations \eqref{twoPhaseEulerCosets}
are equivalent to $\partial_t [\alpha] + [i_u \diff \alpha]  =0$, and therefore to the Euler equation 
$\partial_t [\alpha]+\L_u [\alpha]=0$.

\end{remark}

\begin{proof}[Proof of Theorem \ref{thmMain}]
It suffices to compute $d \H([\balpha], \bmu)$ and apply the Hamiltonian operator. Let $([\balpha](s), \bmu(s))$ be an arbitrary smooth curve in {$\mvect(M)^* = (\Omega^1(M)^n \, / \, \delta(d\Cont^\infty(M))) \times \mdens(M)$} with $[\balpha](0)= [\balpha]$ and $\bmu(0) = \bmu$. Let also $\balpha(s) \in [\balpha](s)$ be the representative  satisfying \eqref{eq:cc}. Then
\begin{gather}
\frac{ \hphantom{}d}{d s}\restrict{s=0}\mathcal H([\balpha](s), \bmu(s)) =\,  \frac{1}{2}\frac{ \hphantom{}d}{d s}\restrict{s=0}   \sum\nolimits_{j=1}^n \int_M (\alpha_j(s), \alpha_j(s)) \mu_j(s) \\ =  \langle \,\vphantom{\sum} \bu,\frac{d}{d s}\restrict{s=0}[\balpha](s)\vphantom{\sum}\, \rangle  \,+\, \frac{1}{2}  \sum\nolimits_{j=1}^n \int_M (\alpha_j(s), \alpha_j(s)) \frac{d}{d s}\restrict{s=0}\mu_j(s)\, ,
\end{gather}
implying that
$$
 {d^F\H}{}([\balpha], \bmu) = \bu, \quad {d^B \H}{}([\balpha], \bmu) = \frac{1}{2}\,[(\balpha, \balpha)]\,.
$$
Now, to get \eqref{twoPhaseEulerCosets}, it suffices to apply the Hamiltonian operator \eqref{HamOperator}, ending the proof.
\end{proof}

\begin{theorem}\label{thmMain2}
The Euler-Arnold equations corresponding to the $L^2$-metric on $\mvect(M)$ written in terms of the fluid velocities $\bu := \I^{-1}([\balpha]) \in \mvect(M)$ read
\begin{subnumcases}{\label{twoPhaseEuler2}}
\partial_t u_j + \nabla_{u_j}u_j = -\grad p,{\label{twoPhaseEuler21}} \\ 
 \qquad \partial_t \mu_j = -\L_{u_j}\mu_j, {\label{twoPhaseEuler22}}
\end{subnumcases}
where the pressure $p \in \Cont^\infty(M)$ is common for all phases and is defined uniquely up to an additive constant by equations \eqref{twoPhaseEuler2} supplemented by the condition
\begin{align}\label{addCond}
\sum\nolimits_{j=1}^n \L_{u_j}\mu_j = 0.
\end{align}
Equivalently, equations \eqref{twoPhaseEuler2} describe the velocity along a geodesic for the $L^2$ metric on a source fiber of the groupoid $\mdiff(M)$.
\end{theorem}

\begin{proof}
Equation \eqref{twoPhaseEuler22} was already established by Theorem \ref{thmMain}, so it suffices to derive  \eqref{twoPhaseEuler21}. The latter rewrites as
$$
\partial_t \alpha_j + i_{u_j}\diff \alpha_j   + \tfrac{1}{2}\diff i_{u_j}\alpha_j   = df\,,
$$
where $\balpha \in [\balpha]$ is the representative satisfying \eqref{eq:cc}, and $f \in \Cont^\infty(M)$ does not depend on $i$. Equivalently, this can be written as
$$
\partial_t \alpha_j + \L_{u_j} \alpha_j   - \tfrac{1}{2}\diff i_{u_j}\alpha_j = df\,.
$$
Taking the metric dual vector field and applying the formula
 $ (\L_u u^\flat - \tfrac{1}{2}\diff (u,u))^\sharp = \nabla_u u$, 
 we get
\begin{equation}\label{twoPhaseEulerDgrad}
 \partial_t \bu + \nabla_{\bu}\bu = \nabla f\,,
\end{equation}
 which is equivalent to \eqref{twoPhaseEuler21} for $p = -f$. %
 
 Now, we show that the pressure $p$ can be expressed, using conditions \eqref{twoPhaseEuler2} and \eqref{addCond}, in terms of velocity fields $u_j$ and densities $\mu_j$ (up to an additive constant). Let $\rho_j := \mu_j / \vol$. Then, from \eqref{twoPhaseEuler21}, we get
 $$
 -\nabla p = -\left(\sum\nolimits_{j=1}^n \rho_j\right) \nabla p =  -\sum\nolimits_{j=1}^n \rho_j \nabla p = \sum\nolimits_{j=1}^n \rho_j(\partial_t u_j + \nabla_{u_j}u_j).
 $$
 Taking divergence, we get
\begin{equation}\label{Poisson1}
-\Delta p = \div \sum\nolimits_{j=1}^n \rho_j(\partial_t u_j + \nabla_{u_j}u_j).
\end{equation}
 Furthermore, \eqref{addCond} can be rewritten as
 $$
 \div \sum\nolimits_{j=1}^n \rho_j u_j = 0,
 $$
 so \eqref{Poisson1} rewrites as
\begin{equation}\label{Poisson2}
 -\Delta p = \div \sum\nolimits_{j=1}^n (\rho_j \nabla_{u_j}u_j - (\partial_t \rho_j) u_j).
 \end{equation}
Also, \eqref{twoPhaseEulerCosets2} is equivalent to
$$
\partial_t \rho_j = -\div(\rho_j u_j),
$$
so \eqref{Poisson2} becomes
   $$
 -\Delta p = \div \sum\nolimits_{j=1}^n (\rho_j \nabla_{u_j}u_j + \div (\rho_j u_j) u_j),
 $$
 cf. \eqref{eq:pe}. This is a Poisson equation on $p$, so the function $p$ is indeed uniquely  determined by $u_j, \mu_j$ up to an additive constant. (Note that for $n=1$ the second term 
 in the right-hand side vanishes and one gets the standard equation for the pressure $-\Delta p = \div  \nabla_{u}u$.)
\end{proof}

Recall that for a fluid velocity field $u$, the corresponding vorticity is the $2$-form $\omega := \diff u^\flat$. For an $n$-tuple of vector fields $\bu \in \mvect(M)$, the vorticity is an $n$-tuple
$$
\omega_j := d u_j^\flat.
$$
\begin{corollary}[Generalized Kelvin's theorem]\label{cor:singKelvin}
For a multiphase fluid, the vorticity of each phase is transported by the corresponding velocity field:
$$
\partial_t \omega_j + \L_{u_j} \omega_j = 0\,.
$$
\end{corollary}
\begin{proof}
Take the exterior derivative of both sides in \eqref{twoPhaseEulerCosets1}.
\end{proof}
In particular, vorticities remain in the same diffeomorphism class during the Euler-Arnold evolution. %
{Furthermore,  solutions for potential initial conditions remain potential for all times, as they correspond to 
the vanishing initial vorticity, which always remains zero thanks to the corollary above. In the next section we discuss 
properties of potential solutions in detail.}
%
%
%
%
%
%
%
%
%
%
%%%%%%%%%%

  \subsection{Potential solutions as geodesics on the space of multiphase densities}
  \label{sect:pure_motion}
  Now, we apply Proposition \ref{prop:sub} to obtain a geodesic description of potential solutions.\par

    \begin{proposition}\label{dualAnchorProp}
    Let $[\bff] \in  \T^*_\bmu \mdens(M)$ (recall that the latter space is {$C_0^\infty(M)^n/\delta(C_0^\infty(M))$}). Then its image under the map $\#^* \colon \T^*_\bmu \mdens(M) \to \mvect(M, \bmu)^*$ is given by 
   \begin{align}\label{dualAnchor2}\#^*[\bff] := [d\bff].\end{align}
   (Note that the coset $[d\bF]$ of $d\bF$ in {$\Omega^1(M)^n / \delta(d\Cont^\infty(M))$} does not depend on the choice of a representative $\bF$ in the coset {$[\bF] \in \Cont_0^\infty(M)^n /\delta( \Cont_0^\infty(M) )$}.)
    \end{proposition}
    \begin{proof}
    Let $[\bff] \in  \T^*_\bmu \mdens(M)$, and let $\bu \in  \mvect(M, \bmu)$. Then
    $$
    \left\langle \#^*[\bff], \bu \right\rangle = \langle [\bff], \#\bu \rangle = -\sum\nolimits_{j=1}^n\int_M  f_j \,\L_{u_j}\,\mu_j=\sum\nolimits_{j=1}^n\int_M  (i_{u_j} df_j) \,\mu_j=\langle d\bff, \bu \rangle\,.
    $$
    The result follows.
    \end{proof}

     It follows that the vector field $\bv: = \I^{-1}(\#^*[\bff])=\I^{-1}([d\bff])$ has the multiphase gradient form $\bv = \nabla \bff:=(\nabla f_1, ..., \nabla f_n)$. This means that the symplectic leaf $\#^*(\T^* \mdens(M)) \subset \mvect(M)^*$ is metric dual to velocity fields of potential motions of
  a multiphase fluid.

   \begin{theorem}\label{geodDescription}

    \begin{enumerate}
    \item Potential solutions of equations \eqref{twoPhaseEuler2} of a multiphase fluid are geodesics of a metric $\metric_{\mdens}$ on $\mdens(M)$ induced by the product Wasserstein metric on the ambient space $\Dens_{c_1}(M) \times \dots \times \Dens_{c_n}(M)$, where $\Dens_c(M)$ is the space of positive smooth densities on $M$ with total mass $c$. 
        \item For any multiphase density $\bmu \in \mdens(M)$ the groupoid target mapping $\Trg \colon (\mdiff(M)_\bmu, \metric_{L^2}) \to (\mdens(M), \metric_\mdens)$ is a Riemannian submersion, see Figure \ref{fig:submersion}. Here $ \metric_{L^2}$ is the restriction of the right-invariant source-wise metric on $\mdiff(M)$ corresponding to the $L^2$-metric on $\mvect(M)$.
    \end{enumerate}
    \end{theorem}

\begin{remark}\label{rem:Was}
Recall that the Wasserstein metric on the space $\Dens_c(M)$ of densities of fixed total volume $c$ on $M$ 
is defined as follows: for any tangent vector $\xi\in T_\mu\Dens_c(M)$ its square length is
$$
      \langle \xi, \xi \rangle_W := \inf \left\{\, \langle u, u \rangle_{L^2} \mid  u \in \Vect(M),\, \L_u\mu = \xi \right\}
      = \langle \nabla f, \nabla f \rangle_{L^2} \,,
$$
where $f \in C^\infty(M)$ is such that $\L_{\nabla f}\mu = \xi $.
\end{remark}

    \begin{proof}[Proof of Theorem \ref{geodDescription}]
    We first prove the existence of a metric $\metric_{\mdens}$  with desired properties, and then show that it is induced by the Wasserstein metric. To prove existence, we use Proposition~\ref{prop:sub}. %
  To apply 
   that proposition we need to show that $\mvect(M) = \Ker \# \oplus (\Ker \#)^\bot$. Take any $\bu \in \mvect(M, \bmu)$.  Consider functions $\bff=(f_1,...,f_n)$ satisfying      $\L_{u_j}\mu_j=\L_{\nabla f_j}\mu_j$ (construction of such functions boils down to the solution of the Poisson equation               ${\rm div}\,{\rho_j} \nabla f_j= {\rm div} \,{\rho_j} {u_j} $ on $f_j$). 
Then one has
\begin{equation}\label{eq:Hodge}
    \bu = (\bu - \nabla \bff) + \nabla\bff.
 \end{equation}
 Notice that $\L_{u_j - \nabla f_j}\mu_j = 0$, so $ \bu - \nabla \bff \in \Ker \# $. Furthermore, $\Ker \# $ consists of all multiphase vector fields $\bu \in \mvect(M,\bmu)$ which satisfy the divergence-free condition $\L_{u_j}\mu_j = 0$ and hence {are} orthogonal  to multiphase gradients with respect to the metric \eqref{eq:metric}. In particular, we have $\nabla\bff \in (\Ker \#)^\bot$. Thus we obtain a decomposition $\mvect(M) = \Ker \# \oplus (\Ker \#)^\bot$, and hence, by Proposition \ref{prop:sub}, a metric  $\metric_{\mdens}$ with the listed properties.
    
   To show that the metric  $\metric_{\mdens}$ is induced by the Wasserstein metric,   observe from \eqref{eq:Hodge} that the map $\#^{-1} \colon \T_\bmu\mdens(M) \to (\Ker \#)^\bot$ is given by $\#^{-1}(\bxi) = \grad \bF$ where the multiphase function $\bF$ is found from the requirement $\xi_j=\L_{\nabla f_j}\mu_j$. Plugging this into \eqref{inducedMetric}, we see that the metric $\metric_{\mdens}$  computed on any $\bxi \in \T \mdens(M)$ is indeed the product Wasserstein metric, as needed.
    \end{proof}

    Recall that $\mdiff(M, \bmu)$ is the configuration space of a multiphase fluid. The motion of the fluid follows the geodesics of the $\metric_{L^2}$-metric on the space $\mdiff(M)$. Potential solutions thus correspond to horizontal (with respect to the target mapping) geodesics.\par
     
 \begin{remark}
    The metric $\metric_\mdens$ constructed above 
    can also be defined as follows: 
    $$
        \langle \bxi, \bxi \rangle_\mdens = \inf \left\{{\langle \bu, \bu \rangle_{L^2} \mid  \bu \in \mvect(M, \bmu), \,\# \bu = \bxi} \right\}\,
        $$
        for any $\bxi\in T_\bmu\mdens(M)$.
        This directly follows from its Riemannian submersion property.
     \end{remark}

\begin{remark}The Riemannian metric on $\mdens(M)$ makes the latter into a metric space with the distance between multiphase densities $\bmu$ and $\tilde\bmu$ satisfying the following inequality:

$$
{\rm dist}_{\mdens}^2(\bmu,\tilde\bmu)\ge \sum_{j=1}^n {\rm dist}_{W}^2(\mu_j,\tilde\mu_j)\,, %
$$
where ${\rm dist}_{W}$ is the Wasserstein distance on $\W(M)$. In particular, the distance function ${\rm dist}_{\mdens}$ is non-degenerate (i.e. ${\rm dist}_{\mdens}(\bmu,\tilde\bmu) >0$ whenever $\bmu \neq \tilde \bmu$).

\end{remark}

\section{Groupoid of generalized flows}\label{sec:GF}

The above consideration can be extended to the case of ``continuous" index $i$, i.e. to multiphase flows where phases are enumerated by a continuous
parameter $a$ which belongs to a measure space $A$. %
Below we adapt all the above definitions and statements to that setting, while the proofs are valid 
{\it mutatis mutandis.} 

Consider a closed compact manifold $M$ with a fixed volume form $\vol$ and a measure space $A$ with a fixed function  $c_a \colon A \to \R$. 
 The base of the Lie groupoid $\gdiff(M)$ of volume-preserving generalized diffeomorphisms is the space $\gdens(M)$ of  generalized densities, i.e. sets of densities $ \bmu:=\{\mu_a \in \Dens(M) \mid a\in A\}$ satisfying the conditions: all $\mu_a$ are positive, have prescribed masses 
 $c_a$, i.e. $ \int_M  \mu_a = c_a$, and they together constitute  the volume form $\vol$, i.e. \begin{equation}\label{eq:tv}\int_A \mu_a \,da = \vol\end{equation} at each point of $M$. ({Here and in what follows we assume that the dependence of all objects on $a \in A$ is such that the integrals below are well-defined. A particular example of such a setting is described in Remark \ref{rm:smooth}.)
Now one can think of those densities as a set $A$ of different  
fractions of an incompressible fluid, penetrating through each other without resistance.

The elements of  $\gdiff(M)$ are sets of diffeomorphisms $\bphi:=\{\phi_a \in \Diff(M) \mid a\in A\}$
of $M$ preserving the above property of incompressibility of generalized densities,
i.e. the set of tuples $(\bphi\,;\bmu, \bmu'):=\{(\phi_a;\mu_a, \mu'_a) \mid a\in A\}$ with generalized forms $\bmu, \bmu'\in \gdens(M)$ such that  $\bphi_*\bmu=\bmu'$ component-wisely, i.e. $\phi_{a*}\mu_a=\mu'_a$ for each $a\in A$. 
The source, target and multiplication (i.e. composition) of such triples is given exactly as before.

Similarly, the space of velocities for a generalized fluid, i.e. the Lie algebroid $\gvect(M) \to \gdens(M)$
corresponding to the Lie groupoid $\gdiff(M)$, is a vector bundle with the following 
structure. Its fiber of $\gvect(M)$ over $\bmu  \in \gdens(M)$ is the space $\gvect(M, \bmu)$ that  consists
 of generalized vector fields on $M$ of the form 
$\bu := \{ u_a \mid a\in A\} $ with  $u_a\in \vect(M)$ that are ``divergence-free" with respect to the generalized volume form: $\int_A \L_{u_a}\mu_a \,da=0$.
The corresponding anchor map  $
 \# \colon \gvect(M, \bmu) \to \T_\bmu \gdens(M)\, $ is given by the negative Lie derivative, $ \bu \mapsto -\L_\bu\bmu:=
\{ -\L_{u_a}\mu_a \mid a\in A\}\,,$ and the algebroid bracket is given by the same formula \eqref{sectionsBracket}.
The tangent space  $\T_\bmu \gdens(M)$ is the space of generalized forms $\bxi$ satisfying the two conditions: 
  $\int_A \xi_a\,da =0$ on $M$ and $\int_M \xi_a = 0$ for all  $a\in A$ (cf. Lemma \ref{lemma:vstanspace}).

\begin{remark}\label{rm:smooth} In the case when $A$ is a manifold and the dependence of all objects on $a \in A$ is smooth, the above setting can also be reformulated as follows.
Consider compact manifolds $M$ and $A$ with fixed volume forms $\vol$ and $\volA$ respectively. The base of the Lie groupoid $\gdiff(M)$ is the space $\gdens(M)$ of \textit{doubly stochastic measures} on $M \times A$ with everywhere positive smooth density, i.e. volume forms $\bmu \in \Dens(M \times A)$ such that $(\pi_M)_*\bmu = \vol$ and $(\pi_A)_*\bmu = \volA$, where $\pi_M$ and $\pi_A$ are projections to $M$ and $A$ respectively. The above description of $\gdens(M)$ is then recovered by viewing a doubly stochastic measure $\bmu$ as a collection of measures $\mu_a$ parametrized by $a \in A$ which have fixed volumes (defined by the measure $\volA$) and add up to the measure $\vol$. %

The elements of the groupoid $\gdiff(M)$ in this language  are horizontal diffeomorphisms $\bphi \in \Diff(M \times A)$ which take one doubly stochastic measure to another. 
More precisely, $\gdiff(M)$ is the set of triples $(\bphi\,;\bmu, \bmu')$ where $\bphi \in \Diff(M \times A)$ is of the form $(x,a) \mapsto (\phi_a(x), a)$ and maps  $\bmu \in \gdens(M)$ to  $\bmu' \in \gdens(M)$, i.e.  $\bphi_*\bmu=\bmu'$. %

The Lie algebroid $\gvect(M) \to \gdens(M)$
corresponding to the Lie groupoid $\gdiff(M)$  is a vector bundle with the following 
structure. Its fiber of $\gvect(M)$ over $\bmu  \in \gdens(M)$ is the space $\gvect(M, \bmu)$ that  consists
 of vector fields  $\bu \in \Vect(M \times A)$ which are horizontal (i.e. tangent to fibers of the projection $\pi_A$) and ``divergence-free" in the sense that $(\pi_M)_* \L_{\bu}\bmu = 0\,. $

The  anchor map  $
 \# \colon \gvect(M, \bmu) \to \T_\bmu \gdens(M)\, $ in the algebroid $\gvect(M)$ is given by the negative Lie derivative, $ \bu \mapsto -\L_\bu\bmu\,,$ and the algebroid bracket is given by the same formula \eqref{sectionsBracket}.
The tangent space  $\T_\bmu \gdens(M)$ is the space of top-degree forms $\bxi$ on $M \times A$ such that $(\pi_M)_*\bxi = 0$ and $(\pi_A)_*\bxi = 0$.

\end{remark}

Returning to the general case of a measure space $A$, the smooth dual  $\gvect(M,\bmu)^*$ of the space $\gvect(M,\bmu)$  is defined as the quotient
$$
\gvect(M,\bmu)^*:=\Omega^1(M)^A/\delta(d\Cont^\infty(M))\,,
$$
where $\Omega^1(M)^A$  stands for functions $A \to \Omega^1(M)$. 
The elements of $\gvect(M,\bmu)^*$ are cosets 
$$
[\balpha]:=\{\alpha_a +df \mid f\in \Cont^\infty(M) \},
$$
where all $1$-forms $\alpha_a\in \Omega^1(M)$ in one coset differ by the same function differential.
The pairing between a coset $[\balpha] \in  \gvect(M, \bmu)^*$ and a generalized vector field $\bu  \in   \gvect(M,\bmu)$ is given by the formula: $$\langle[\balpha], \bu\rangle :=  \int_A  \int_M  \alpha_a(u_a)\,  \mu_a\, da.$$
As before, the dual algebroid is the  total space
 $$
 \gvect(M)^*\,:= \bigcup\nolimits_{\bmu \in \gdens(M)} \gvect(M,\bmu)^*\,,
 $$ 
which is a trivial vector bundle  over the space of generalized densities $\gdens(M)$.

The dual algebroid is a Poisson bundle, and the  Poisson bracket on this space is given by a formula analogous to \eqref{pbf}.
As a corollary, we obtain the  Hamiltonian operator %
 $$
 \P_{[\balpha], \,\bmu}^\sharp \colon \T^*_{[\balpha], \,\bmu}\gvect(M)^* \to  \T_{[\balpha], \,\bmu}\gvect(M)^*
 $$
 corresponding to the Poisson bracket on $\gvect(M)^* $  given by the same formula \eqref{HamOperator}:
$$
  (\bu, [\bF]) \mapsto (-[i_{\bu}\diff \balpha] - [d\bF], -\L_\bu \bmu)\,.
$$

To describe geodesics on the space  of generalized solutions we equip the manifold $M$ with a Riemannian metric $(\,,)$ whose Riemannian 
volume form is $\vol$. As before, to simplify the exposition, $M$ is a compact connected oriented manifold without boundary, although the results 
extend to noncompact $M$ by imposing appropriate decay assumptions. 

The  $\metric_{L^2}$  metric on the Lie algebroid $\gvect(M)$ is as follows: for $\bu, \bv \in \gvect(M, \bmu) $, one has
\begin{equation}\label{eq:pairing}
\langle \bu, \bv \rangle_{L^2} := \int_A \int_M (u_a, v_a) \,\mu_a\,da.
\end{equation}
The inertia operator $\I: \gvect(M)\to \gvect(M)^*$ associated with this $L^2$-metric  on $\gvect(M)$ is as follows.
For $\bu \in \gvect(M, \bmu)$ one has
$
\I(\bu) = [\bu^\flat]
$,
where $\bu^\flat := \{ u_a^\flat \mid a\in A \}$. Here  $u_a^\flat$ is the 1-form metric-dual to the vector field $u_a$ on $M$, and $[\bu^\flat]$ stands for the coset of $\bu^\flat$ in  $ \Omega^1(M)^A \, / \, \delta(d\Cont^\infty(M)) $.

The corresponding Euler-Arnold Hamiltonian on  $\gvect(M)^*$ is 
$$
\mathcal H\left([\balpha], \bmu\vphantom{}\right) := \tfrac{1}{2} \int_{A} \int_M (\alpha_a, \alpha_a) \,\mu_a \,da,
$$
where $\balpha \in [\balpha]$ is a representative satisfying the following co-closedness type condition:
\begin{equation}\label{eq:cc2}d^*\int_A\rho_a \alpha_a\,da = 0\,,\end{equation}
for $\rho_a : = \mu_a / \vol$ {(this is just the condition \eqref{eq:nde} written in terms of the multiform $\balpha = \bu^\flat$)}.
 With this adjustment of notations the following theorem literally repeats 
Theorem \ref{thmMain} and provides the Hamiltonian framework for generalized flows:

\begin{theorem}\label{thmMain3}
The Euler-Arnold equation for generalized flows corresponding to the $L^2$-metric on $\gvect(M)$ written in terms of a coset $[\balpha] \in \gvect(M, \bmu)^*$ reads
\begin{align}\label{twoPhaseEulerCosets3}
\begin{cases}
\partial_t [\balpha] + [i_\bu \diff \balpha + \tfrac{1}{2}\diff i_\bu\balpha]  =0\,,\\
\qquad \partial_t \bmu = -\L_\bu \bmu\,,
\end{cases}
\end{align}
where $\balpha \in [\balpha]$ is the representative satisfying \eqref{eq:cc2}, and $\bu := \balpha^\sharp \in \gvect(M, \bmu)$. 
It is a Hamiltonian equation on the algebroid dual $\gvect(M)^*$ with respect to the natural Poisson structure described above and the energy Hamiltonian function~$\mathcal H$.
\end{theorem}

Let us rewrite explicitly the Euler-Arnold equations in terms of fluid velocities of generalized flows.

\begin{theorem}\label{thmMain4}
The Euler-Arnold equations corresponding to the $L^2$-metric on $\gvect(M)$ written in terms of the fluid velocities $\bu$ and density $\rho_a = \mu_a / \vol$ coincide with generalized flow equations (1.13) - (1.15) of \cite{brenier2}:
{
\begin{subnumcases}{}%
%\int_A \rho_a \,da = 1,\label{eq:b1}\\
\partial_t (\rho_a u_a) + \div(\rho_a u_a \otimes u_a) + \rho_a \grad p = 0,\label{eq:b3}\\
\partial_t \rho_a + \div(\rho_a u_a ) = 0,\label{eq:b2}
\end{subnumcases}
subject to the constraint $\int_A \rho_a \,da = 1,$}
where the pressure $p \in \Cont^\infty(M)$ is common for all phases and is defined uniquely up to an additive constant by these equations.

Equivalently, the generalized flow equations \eqref{eq:b3}-\eqref{eq:b2} describe the velocity along a geodesic for the $L^2$ metric on a source fiber of the groupoid $\gdiff(M)$.
\end{theorem}

\begin{proof}
The constraint $\int_A \rho_a \,da = 1$  %Equation \eqref{eq:b1} 
follows from \eqref{eq:tv} and the definition  $\rho_a = \mu_a / \vol$ of the density function. %So it suffices to derive equations \eqref{eq:b3}  and \eqref{eq:b2}.  
Using Theorem \ref{thmMain3} and the same argument as in the proof of Theorem \ref{thmMain2}, one gets the equations
\begin{subnumcases}{}
\partial_t u_a + \nabla_{u_a}u_a = -\grad p \label{eq:b4}, \\ 
 \qquad \partial_t \mu_a = -\L_{u_a}\mu_a \label{eq:b5}.
\end{subnumcases}
Clearly, \eqref{eq:b5} is equivalent to \eqref{eq:b2}, so it only remains to derive \eqref{eq:b3}. Using \eqref{eq:b2} and \eqref{eq:b4}, we obtain
\begin{align}
\partial_t (\rho_a u_a) &= (\partial_t \rho_a) u_a + \rho_a (\partial_t u_a) = -\div(\rho_a u_a) u_a - \rho_a(\nabla_{u_a}u_a + \grad p)
\\ &= -(\div(\rho_a u_a) u_a +\nabla_{\rho_au_a}u_a ) -\rho_a \grad p = -\div(\rho_a u_a \otimes u_a) - \rho_a \grad p\,,
\end{align}
where the last equality follows from the identity $\div(u \otimes v) = \div(u)v + \nabla_u v$. Thus, \eqref{eq:b3} is equivalent to \eqref{eq:b4}, as required.
\end{proof}

\begin{remark}
The quantity $\pmb{m}:=\pmb{\rho} \bu^\flat=\{\rho_a u_a^\flat \mid \,a\in A\}$ has the physical meaning of the momentum. 
In terms of momentum the metric \eqref{eq:pairing} 
assumes a simpler form  $\langle \bu, \bv \rangle_{L^2} = \int_A \int_M m_a(v_a)\, \vol \,da$ and the equations \eqref{eq:b3}-\eqref{eq:b2} are often written on $m_a$, cf.~\cite{brenier2}.
\end{remark}

%%%%%%%%%%%%%%%%%%%
 {
\section{Open problems}\label{sec:open}
Arnold's original insight in \cite{Arn66} uncovered the geometry behind the hydrodynamic Euler equation:  for an ideal fluid confined to a fixed domain the Euler equation describes the geodesic flow for the energy metric on the Lie group of volume-preserving 
diffeomorphisms of  that domain. The analytical part of this approach is due to Ebin and Marsden~\cite{EbMars70}
who proved short-time existence in the setting of Sobolev spaces $H^s$, where $s$ is sufficiently large ($s>\dim M/2+1$).

The present paper can be regarded as an analog of Arnold's take by providing the geometric framework of Lie groupoids, instead of Lie groups,
for the Euler equation for multiphase fluids and generalized flows. We hope that it will encourage the appearance of necessary analytical setting, in the form of existence theorems in appropriate Sobolev or tame Fr\'echet spaces. 

Here we summarize several open problems motivated by the groupoid approach:

\smallskip

-- Provide an analytic framework and existence theorems for the Euler equation for multiphase fluids and generalized flows, extending the Ebin--Marsden setting \cite{EbMars70} from groups to groupoids of diffeomorphisms.

\smallskip

-- There are (at least) two different definitions of generalized flows, both suggested by Y.~Brenier: the one discussed above, as a continuum version of multiphase flows \cite{brenier, brenier2}, and the other via 
 probabilistic measures on the space of all parametrized continuous paths $ X =~C([0, 1];M)$ satisfying the incompressibility and finiteness of action conditions  \cite{brenier0}, see also~\cite{AK, Sh2}. 
Their equivalence is intuitively assumed but, to the best of our knowledge, not written up. Once it is formally established, it would open  new ways of applying groupoids in probabilistic settings. 

\smallskip

-- There is a natural semigroup of continuous maps, in which fluid particles are allowed to collide and stick to each other.
In that setting compositions of maps are well-defined but inversion is not, cf. \cite{brenier2}, which seems to be an appropriate framework for the description of shock waves in fluids.  While there seem to be a projection from the 
diffeomorphism groupoid to the semigroup of maps, the corresponding Hamiltonian picture for the semigroup is rather obscure.

\smallskip

-- In the appendix below we give a groupoid description of vortex sheets \cite{IK}, which can be thought of as the limiting case for the relaxed problem of evolution of homogenized vortex sheets or miltiphase flows, see \cite{brenier, Loesch}.
It would be interesting to obtain a rigorous treatment of this limiting procedure in the Lagrangian and Hamiltonian setting.

\smallskip

-- Finally, it would be interesting to apply the framework of Euler-Arnold equations on Lie groupoids, along with the corresponding Hamiltonian framework on Lie algebroid duals, to other problems in mathematical physics, both in finite and infinite dimensions. This approach seems natural in the situations where the group symmetry is not available, e.g. fluids with dynamic boundary, a rigid body moving in a manifold, etc.
}

%%%%%%%%%%%%%

  \section{Appendix: Dynamics of classical vortex sheets}
  \refstepcounter{AppCounter}
\label{appA}

Classical vortex sheets can be thought of as a particular case (or, rather, as belonging to a closure) of multiphase fluids
 where the densities are  indicator functions of open sets separated by a hypersurface in a manifold  $M$, see \cite{Loesch, IK}. 
Namely, the multiphase Lie groupoid in that case becomes  the Lie  groupoid $\DSDiff(M)$ of volume-preserving diffeomorphisms 
of $M$ that are discontinuous along a hypersurface. The elements of  the groupoid
$\DSDiff(M)$ are  quadruples $(\Sheet_1, \Sheet_2, \phi^+, \phi^-)$, where 
$\Sheet_1, \Sheet_2 \in \VS(M) $ are hypersurfaces (vortex sheets) in $M$ confining the same total volume, while  $\phi^\pm \colon \Dom^\pm_{\Sheet_1} \to \Dom^\pm_{\Sheet_2}$ 
are volume-preserving diffeomorphisms between  connected  components of $M\, \setminus \, \Sheet_i$ 
denoted by $\Dom_{\Sheet_i}^+, \Dom_{\Sheet_i}^-$.  
The multiplication of the quadruples in $\DSDiff(M)$ is given by the natural composition of discontinuous diffeomorphisms and is shown in Figure \ref{fig:composition}.
\medskip

 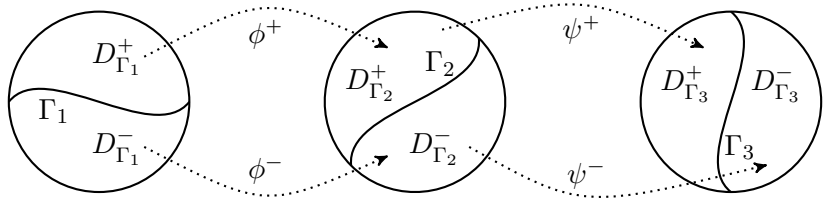
\begin{figure}[b]
\centerline{
\begin{tikzpicture}[thick, scale = 1.2]
 \node  at (0,0) () {
 \begin{tikzpicture}[thick, scale = 1.2]
    \draw (1, 1) ellipse (1cm and 1cm);
     \draw (0,1)  .. controls (0.4,1.5) and  (1.6,0.5) ..  (2,1);
      \node  at (0.5,0.9) () {$\Sheet_1$};
        \node  at (1.2,1.5) () {$\Dom^+_{\Sheet_1}$};
             \node  at (1.2,0.5) () {$\Dom^-_{\Sheet_1}$};
     \end{tikzpicture}
     };
      \node  at (3.5,0) () {
 \begin{tikzpicture}[thick, rotate = 45, scale = 1.2]
    \draw (1, 1) ellipse (1cm and 1cm);
     \draw (0,1)  .. controls (0.4,1.5) and  (1.6,0.5) ..  (2,1);
        \node  at (1.5,1.1) () {$\Sheet_2$};
                \node  at (0.8,1.5) () {$\Dom^+_{\Sheet_2}$};
             \node  at (0.8,0.5) () {$\Dom^-_{\Sheet_2}$};
     \end{tikzpicture}
     };
           \node  at (7,0) () {
 \begin{tikzpicture}[thick, rotate = 90, scale = 1.2]
    \draw (1, 1) ellipse (1cm and 1cm);
     \draw (0,1)  .. controls (0.4,1.5) and  (1.6,0.5) ..  (2,1);
          \node  at (0.5,0.9) () {$\Sheet_3$};
                  \node  at (1.2,1.5) () {$\Dom^+_{\Sheet_3}$};
             \node  at (1.2,0.5) () {$\Dom^-_{\Sheet_3}$};
     \end{tikzpicture}
     };
     \draw [dotted, ->] (0.5,0.5) .. controls (1.85, 1.2) .. (3.2,0.6);
          \draw [dotted, ->] (0.5,-0.5) .. controls (1.85, -1.2) .. (3.2,-0.6);
            \node  at (1.85,0.75) () {$\phi^+$};
                     \node  at (1.85,-0.75) () {$\phi^-$};
                          \draw [dotted, ->] (3.8,0.8) .. controls (5.35, 1.2) .. (6.7,0.6);
            \node  at (5.35,0.8) () {$\psi^+$};
                    \draw [dotted, ->] (4.1,-0.5) .. controls (5.35, -1.2) .. (7.4,-0.7);
                      \node  at (5.4,-0.8) () {$\psi^-$};
\end{tikzpicture}
}
\caption{Elements of the groupoid  $\DSDiff(M)$ and their composition rule.}\label{fig:composition}
\end{figure}

\smallskip

The corresponding Lie algebroid $\dsvect(M) \to \VS(M)$ 
is the space of possible velocities of a fluid with a vortex sheet, defined as follows.
Given a vortex sheet $\Sheet$, the corresponding velocities are discontinuous vector fields on $M$ 
of the form $u = \chiplus u^+ + \chimin u^-\!,$
where $\chiplus, \chimin$ are indicator functions of the connected components $\Dompm$ 
of $M\, \setminus \,\Sheet$, and $u^\pm$ are smooth divergence-free vector fields on $\Dompm$ which have the same normal component on $\Sheet$.
The map from such  vector fields $u$ to their normal components on $\Sheet$  is the anchor map $\#$
of the corresponding algebroid.
Via the general procedure described above one defines a right-invariant $L^2$-metric on this groupoid and constructs an analogue of the geodesic Euler-Arnold equation. 

\begin{theorem}[\cite{IK}]
The  Euler-Arnold equation  corresponding to the 
$L^2$-metric on the algebroid $\dsvect(M)$ coincides with the 
 the Euler equation  for a fluid flow discontinuous along a vortex sheet $\Sheet\subset M$:
\begin{align}\label{twoPhaseEuler}
\begin{cases}
\partial_t u^+ + \nabla_{u^+} u^+ = -\grad p^+\!, \\ 
\partial_t u^- + \nabla_{u^-} u^-  = -\grad p^-\!,\\
 \qquad \partial_t \Sheet = \# u\,,
\end{cases}
\end{align}
where $u = \chiplus u^+ + \chimin u^-$ is the fluid velocity, ${\rm div}\, u^\pm=0$, and 
$p^\pm \in \Cont^\infty(\Dompm)$ are functions satisfying the continuity condition
$p^+\vert_{\Sheet} = p^-\vert_{\Sheet}$.

Equivalently, Euler equations \eqref{twoPhaseEuler} are geodesic
equations for the right-invariant $L^2$-metric on (source fibers of) the Lie groupoid $\DSDiff(M)$ 
of discontinuous volume-preserving diffeomorphisms.
\end{theorem}

The above consideration also defines a metric on the space $\VS(M)$ of vortex sheets, while the 
target map is a Riemannian submersion of the $L^2$-metric on the groupoid of discontinuous diffeomorphisms to the metric on $\VS(M)$,
see \cite{Loesch, IK}.

 {
\begin{remark}
The multiphase groupoid studied in Section \ref{sect:diffeo_pair}
can be regarded as a relaxed version of the vortex sheet groupoid as follows.
Given a hypersurface $\Sheet\subset M$ we define the multiphase density 
$ \bmu:=(\mu_+,\mu_-)$ as a pair of indicator densities $\mu_\pm=\chi_\pm\vol$ for 
indicator functions  $\chi_\pm$  of the connected components $\Dompm$ 
of $M\, \setminus \,\Sheet$, and therefore satisfying the condition $\mu_+ + \mu_-=\vol$ on $M$.
Now, for a groupoid element $ (\bphi\,;\bmu, \bmu'):=(\phi_+, \phi_-;\mu_+,\mu_-, \mu_+',\mu_-')$
 the pair of diffeomorphisms $\phi_\pm$ on $M$ satisfying ${\phi_\pm}_*\mu_\pm=\mu_\pm'$ for 
indicator densities $\mu_\pm'$ representing the connected components $D^\pm_{\Sheet'}$ of $M\, \setminus \,\Sheet'$
 boils down to a pair of $\vol$-preserving diffeomorphisms sending, respectively, $\Dompm$ to $D^\pm_{\Sheet'}$, i.e.
 $\Sheet$ to $\Sheet'$ while preserving the volume form $\vol$ on $M$.
 
 One can see that the definitions of the corresponding algebroids, their brackets and anchor maps, as well as the corresponding Poisson structures and Hamiltonian equations are consistent with taking this relaxed version and lead
 to the relation between the Euler equations~\eqref{eq:hvse} and~\eqref{twoPhaseEuler}. It would be 
 interesting to formally establish the convergence for the relaxed solutions to the classical solutions with vortex sheets,
 cf. Section \ref{sec:open}.
 \end{remark}
}

%%%%%%%%%%%%%%

\bibliographystyle{plain}
\bibliography{vs}

\end{document}